\documentclass[11pt]{article}

\usepackage[utf8]{inputenc} 
\usepackage{amsmath}
\usepackage{graphicx}
\usepackage{subcaption}
\usepackage[top=30mm, bottom=30mm, left=27mm, right=27mm]{geometry}
\usepackage{amsfonts}
\usepackage{mathtools}
\usepackage{subcaption}
\usepackage{amssymb}
\usepackage{amsthm}
\usepackage{chngpage}
\usepackage{hyperref}
\hypersetup{
	colorlinks=true,
	linkcolor=black,
	citecolor=black,
	filecolor=black,      
	urlcolor=black,
}
\usepackage{enumerate}
\usepackage{makeidx}
\usepackage{bm}
\usepackage{thmtools}
\usepackage{thm-restate}

\makeindex

\makeatletter
\renewenvironment{proof}[1][\proofname] {\par\pushQED{\qed}\normalfont\topsep6\p@\@plus6\p@\relax\trivlist\item[\hskip\labelsep\bfseries#1\@addpunct{.}]\ignorespaces}{\popQED\endtrivlist\@endpefalse}
\makeatother

\renewcommand{\AA}{\mathcal{A}}

\newcommand{\CC}{\mathcal{C}}

\newcommand{\JJ}{\mathcal{J}}

\newcommand{\RR}{\mathbb{R}}

\newcommand{\SO}{\operatorname{SO}}

\newcommand{\Sph}{\mathbb{S}}

\newtheorem{proposition}{Proposition}[section]

\newtheorem{lemma}[proposition]{Lemma}
\newtheorem{theorem}[proposition]{Theorem}

\newtheorem{question}[proposition]{Question}
\usepackage{amsthm}

\theoremstyle{definition}

\newtheorem*{remark*}{Remark}
\newtheorem*{theorem*}{Theorem}

\mathcode`l="8000
\begingroup
\makeatletter
\lccode`\~=`\l
\DeclareMathSymbol{\lsb@l}{\mathalpha}{letters}{`l}
\lowercase{\gdef~{\ifnum\the\mathgroup=\m@ne \ell \else \lsb@l \fi}}%
\endgroup

\title{Rotation inside convex Kakeya sets}
\author{Barnabás Janzer\thanks{Department of Pure Mathematics and Mathematical Statistics, University of Cambridge, Wilberforce Road, Cambridge CB3 0WB, United Kingdom. Email: bkj21@cam.ac.uk. This work was supported by EPSRC DTG.}
}
\date{\vspace{-21pt}}

\begin{document}
	\maketitle
\begin{abstract}
	Let $K$ be a convex body (a compact convex set) in $\RR^d$, that contains a copy of another body $S$ in every possible orientation. Is it always possible to continuously move any one copy of $S$ into another, inside $K$? As a stronger question, is it always possible to continuously select, for each orientation, one copy of $S$ in that orientation? These questions were asked by Croft.
	
	We show that, in two dimensions, the stronger question always has an affirmative answer. We also show that in three dimensions the answer is negative, even for the case when $S$ is a line segment -- but that in any dimension the first question has a positive answer when $S$ is a line segment. And we prove that, surprisingly, the answer to the first question is negative in dimension four for general $S$.
\end{abstract}
\section{Introduction}
A subset $K$ of $\RR^d$ is called a \emph{Kakeya set} (or \emph{Besicovitch set}) if it contains a unit segment in all directions, i.e., whenever $v\in \Sph^{d-1}$ then there is some $w\in K$ such that $w+tv\in K$ for all $t\in[0,1]$. Some of the earliest results about Kakeya sets were proved by Besicovitch \cite{besicovitch1,besicovitch1928kakeya}, who proved that there exist Kakeya sets of measure zero, and also showed that there are (Kakeya) sets in $\RR^2$ of arbitrarily small measure in which a unit segment can be continuously moved and rotated around by $360^\circ$. Since then there has been a lot of interest in Kakeya sets and related problems, see, e.g., \cite{davies_1971,bourgain1999dimension,kolasa1999some,katz2002new}. The study of Kakeya sets is connected to surprisingly many different areas of mathematics, including harmonic analysis, arithmetic combinatorics and PDEs (see, e.g.,~\cite{fefferman1971multiplier,bourgain1999dimension}). One of the most interesting open problems about Kakeya sets is the Kakeya conjecture, which claims that if $K$ is a compact Kakeya set in $\RR^d$, then $K$ has (Hausdorff) dimension $d$ (see, e.g., \cite{bourgain1999dimension}).


While the conjecture above is probably the most important open problem in the area, there has also been much interest in questions about Kakeya sets that are more similar to the original problem, and study when we can rotate a unit segment around inside another body. For example, van~Alphen~\cite{vanalphen} showed that it is possible to construct sets of arbitrarily small area and bounded diameter in $\RR^2$ in which a segment can be rotated around. Cunningham~\cite{cunningham1971kakeya} showed that such a set can even be made simply connected. Csörnyei, Héra and Laczkovich~\cite{csornyei2017closed} showed that if $S$ is a closed and connected set in $\RR^2$ such that any two copies of $S$ can be moved into each other within a set of arbitrarily small measure, then $S$ must be a segment, a circular arc, or a singleton. Järvenpää, Järvenpää, Keleti and Máthé~\cite{jarvenpaa2011continuously} proved that for $n\geq 3$ it is possible to move a line around within a set of measure zero in $\RR^n$ such that all directions are traversed; however, if $K\subseteq \RR^n$ is  such that we can choose a copy of a line in each direction simultaneously in a continuous way (parametrized by $\Sph^{n-1}$), then the complement of $K$ must be bounded. There is a very large literature on Kakeya sets, and many other interesting problems have been studied, see, e.g.,~\cite{dvir2009size,hera2016kakeya,falconer1980continuity}.

As hinted above, several results about Kakeya sets concern the stronger property of being able to continuously move and rotate around a segment (or some other set), as opposed to simply containing a segment in each direction (i.e., being Kakeya). It is then interesting to ask how strong the former property is compared to the latter: can we make some additional, natural assumption on our set such that the second property implies the first one? Without any such assumptions, being Kakeya does not imply the first property -- for example, our set could consist of two, disconnected components that together cover all possible orientations of segments. It is also easy to see that being connected is not enough -- but what happens if our set is \emph{convex}? This question, in the following more general form, was asked by Croft~\cite{hallardcroft}.

\begin{question}[Croft \cite{hallardcroft}]\label{question_kakeya}
	If $K$ is a convex and compact set in $\RR^d$ that contains a copy of $S\subseteq \RR^d$ in every possible orientation, 
	is it necessarily possible to continuously transform any given copy of $S$ into any other one within $K$?
\end{question}

While it is very natural to study \emph{convex} Kakeya sets, and they were already considered over a hundred years ago by Pál~\cite{pal1921minimumproblem} (who proved that the minimal possible area of a convex Kakeya set in $\RR^2$ is $1/\sqrt{3}$), it is important to point out that the question above is of a different flavour. Indeed, apart from focusing on convex sets, a significant difference between Question~\ref{question_kakeya} and most of the known results about Kakeya sets is that here we are not interested in the measure of our Kakeya set (unlike in the papers mentioned earlier).

To formalise Question~\ref{question_kakeya}, we first need some definitions. For any set $S\subseteq \RR^d$, let us say that $K\subseteq \RR^d$ is \emph{$S$-Kakeya} if $K$ contains a translate of any rotated copy of $S$, i.e., whenever $\rho\in \SO(d)$ then there is some $w\in \RR^d$ such that $\rho(S)+w\subseteq K$. In particular, when $S$ is a segment of length $1$ then this is just the usual notion of being a Kakeya set. 
Let us also say that \emph{any two $S$-copies can be rotated into each other within $K$} if whenever $\rho_0,\rho_1\in \SO(d)$ and $w_0,w_1\in \RR^d$ are such that $\rho_i(S)+w_i\subseteq K$ ($i=0,1$), then there are some $\gamma:[0,1]\to \SO(d)$ and $\delta: [0,1]\to \RR^d$ continuous functions such that $\gamma(i)=\rho_i, \delta(i)=w_i$ for $i=0,1$ and $\gamma(t)(S)+\delta(t)\subseteq K$ for all $t$.

We mention that instead of having continuous $\gamma, \delta$ as above, we could define this notion in terms of a single continuous function $\psi$ mapping each $t\in [0,1]$ to a (rotated and translated) copy of $S$ in a continuous way (with respect to the Hausdorff metric), our results below still hold in this alternative characterisation. Furthermore, in the case of usual Kakeya sets (i.e., when $S$ is a unit segment), we can also parametrize the possible orientations of segments by the sphere $\Sph^{d-1}$ or by the projective space $\mathbb{P}\RR^d$ instead of $\SO(d)$, but these changes would make no difference.\medskip

Our first result shows that in the case of usual Kakeya sets, any two unit segments can be rotated into each other within $K$ (if $K$ is convex and compact).
\begin{theorem}\label{theorem_Kakeyasegments}
	Let $d\geq 2$ be a positive integer and let $K$ be a convex Kakeya body in $\RR^d$. Then any two unit segments can be rotated into each other within $K$.
\end{theorem}

Given Theorem~\ref{theorem_Kakeyasegments}, one might expect that the corresponding statement is in fact true for any set $S$. Surprisingly, this is not the case.

\begin{restatable}{theorem}{fourdimcounterexample}\label{theorem_counterexamplereachable_introduction}
	There exist convex bodies $S$ and $K$ in $\RR^4$ such that $K$ is $S$-Kakeya but there are two $S$-copies which cannot be rotated into each other within $K$.
\end{restatable}

While the result above is stated for $d=4$, it is in fact easy to modify our construction to get a counterexample for any $d\geq 4$. 

We mention that, in contrast with Theorem~\ref{theorem_counterexamplereachable_introduction}, if we replace the assumption `$K$ compact' by `$K$ open', then an easy connectedness argument shows that any two copies can be rotated into each other.\medskip

An alternative way to interpret Question~\ref{question_kakeya} is to ask for a way to select a copy of $S$ (in an $S$-Kakeya set) in each direction simultaneously in a continuous way. That is, we want the stronger property that there exists a continuous map $f: \SO(d)\to \RR^d$ such that $\rho(S)+f(\rho)\subseteq K$ for all $\rho$. We show that this can be achieved for any shape in $2$ dimensions.
\begin{theorem}\label{theorem_continuousin2d}
	Let $K$ be a convex body in $\RR^2$ and let $S\subseteq \RR^2$, $S\not =\emptyset$. Assume that $K$ is $S$-Kakeya. Then there is a continuous map $f: \SO(2)\to \RR^2$ such that $\rho(S)+f(\rho)\subseteq K$ for all $\rho\in \SO(2)$.
\end{theorem}

Again, in light of Theorem~\ref{theorem_continuousin2d}, one might expect that the corresponding statement is true in higher dimensions too, at least when $S$ is a line segment. However, this strong property fails already when $d=3$, even when $S$ is a unit segment.
\begin{theorem}\label{theorem_continuouscounterexample}
	There exists a convex Kakeya body $K\subseteq \RR^3$ such that there is no continuous function $\psi: \Sph^2\to \RR^3$ satisfying $\psi(v)+tv\in K$ for all $v\in \Sph^2, t\in [0,1]$.
\end{theorem}

As before, the fact that we chose to parametrize orientations of segments by the sphere $\Sph^{d-1}$ instead of $SO(d)$ or the projective space $\mathbb{P}\RR^d$ does not change anything, Theorem~\ref{theorem_continuouscounterexample} would remain true for these parametrizations as well, and reason why the counterexample works is not topological.

The rest of the paper is organised as follows. In Section~\ref{section_2d}, we prove Theorem~\ref{theorem_continuousin2d} and Theorem~\ref{theorem_continuouscounterexample} concerning the stronger property of being able to continuously select in all directions. In Section~\ref{section_Kakeyareachable}, we prove Theorem~\ref{theorem_Kakeyasegments} about rotating in convex Kakeya sets in $\RR^d$ for any $d$, and in Section~\ref{section_reachablecounterexample} we give a counterexample for the corresponding statement for general bodies. We finish with some concluding remarks and open questions in Section~\ref{section_Kakeyaconcluding}.

The proofs in Section~\ref{section_2d} are simpler than the ones in the later sections, but several elements of those proofs reappear or motivate our later approach. In particular, one of the main methods we will have for analysing different cases is to consider the dimensions of the sets $I_\rho=\{w\in \RR^d: \rho(S)+w\subseteq K\}$. It is easy to deal with $\rho$ (and its neighbourhood) if $I_\rho$ has dimension $d$ (i.e., has non-empty interior). One might initially expect that the larger the dimension of $I_\rho$ is, the more room we have to move the copies around and hence the easier to deal with $\rho$. However, this is not entirely true, and the $0$-dimensional case (when $I_\rho$ is a single point) will be quite easy to deal with. For example, it is not difficult to prove that if $I_\rho=\{w_\rho\}$ is a single point for all $\rho$, then $\rho\mapsto w_\rho$ must be continuous. So the most difficult cases in Theorem~\ref{theorem_Kakeyasegments} will come from the situation when some $I_\rho$ has dimension between $1$ and $d-1$, and these will also be the cases we use to obtain counterexamples in Theorems~\ref{theorem_continuouscounterexample} and \ref{theorem_counterexamplereachable_introduction}.

\section{Continuous choice in each direction}\label{section_2d}
In this section we prove Theorem~\ref{theorem_continuousin2d} and Theorem~\ref{theorem_continuouscounterexample} about selecting a copy in each direction in a continuous way. We begin with Theorem~\ref{theorem_continuousin2d}.

First we recall the definition of the Hausdorff metric. Given a point $p\in \RR^d$ and a non-empty compact set $A\subseteq \RR^d$, write $$d(p,A)=\min_{a\in A}|p-a|.$$ Given two non-empty compact sets $X,Y\subseteq \RR^d$, their distance in the Hausdorff metric $d$ is defined as

$$d(X,Y)=\max\{\max_{x\in X}d(x,Y),\max_{y\in Y}d(y,X)\}.$$
It is well-known that this makes the set $\CC_d$ of non-empty compact subsets of $\RR^d$ a metric space. Let $\mathcal{K}_d$ denote the set of non-empty compact convex sets in $\RR^d$ (so $\mathcal{K}_d\subseteq\mathcal{C}_d$).

We will prove the following result.

\begin{lemma}\label{lemma_Hausdorffcontinuous}
	Let $S$ be a non-empty compact subset of $\RR^2$ and let $K$ be convex, compact and $S$-Kakeya. For all $\rho\in \SO(2)$, let $I_{\rho}=\{v\in\RR^2:\rho(S)+v\subseteq K\}$. Then the map $\SO(2)\to \mathcal{K}_2$ given by $\rho\mapsto I_{\rho}$ is continuous.
\end{lemma}

Given $I\in \mathcal{K}_d$, we say that $I$ has Chebyshev centre $c$ if $x=c$ minimises $\max_{p\in I} |x-p|$ among all points $x\in\RR^d$. We will use the following properties of Chebyshev centres. (Much more general statements are known about Chebyshev centres in Banach spaces, but the next result is enough for our purposes.)

\begin{lemma}\label{lemma_chebyshevproperties}(See, e.g., \cite[Theorem~5]{amir1978chebyshev} and \cite[subsection 7.1]{amir1984best})
	If $I\in \mathcal{K}_d$ then $I$ has a unique Chebyshev centre $c_I$. Moreover, $c_I\in I$ for all $I$, and the map $\mathcal{K}_d\to\RR^d$ given by $I\mapsto c_I$ is continuous.
\end{lemma}

It is easy to see that Theorem~\ref{theorem_continuousin2d} follows from Lemmas~\ref{lemma_Hausdorffcontinuous} and \ref{lemma_chebyshevproperties}. So we now need to prove Lemma~\ref{lemma_Hausdorffcontinuous}. 
In fact, we will prove the following stronger statement.
\begin{lemma}\label{lemma_2dallhausdorff}
	Let $K$ be a compact convex set in $\RR^2$. For any non-empty compact set $S$ in $\RR^2$, let $I_S=\{w\in \RR^2: S+w\subseteq K\}$. Let $\mathcal{A}_K$ be the set of all $S$ with $I_S$ non-empty. Then the map $\psi:\mathcal{A}_K\to \mathcal{K}_2$ given by $S\mapsto I_S$ is continuous (with respect to the Haudorff metric on both sides).
\end{lemma}

Lemma~\ref{lemma_2dallhausdorff} certainly implies Lemma~\ref{lemma_Hausdorffcontinuous}, as $\rho\mapsto \rho(S)$ is easily seen to be continuous for any fixed $S$. Also, note that  Lemma~\ref{lemma_Hausdorffcontinuous} and Lemma~\ref{lemma_2dallhausdorff} are not true in dimensions greater than $2$, by the construction in Theorem~\ref{theorem_continuouscounterexample}.

Let us start the proof of Lemma~\ref{lemma_2dallhausdorff}. The first lemma towards the proof essentially says that if $I_S$ is a segment on the $x$ axis (so $I_S$ is one-dimensional), then the projections of $K$ and $S$ to the $y$ axis have the same maximum values (and similarly minimum values). This is rather easy to see when $S$ is a segment, and only slightly more complicated in general.

\begin{lemma}\label{lemma_2dsameprojections}
	Suppose that $K\subseteq\RR^2$ is compact and convex, $S\subseteq \RR^2$ is non-empty and compact, and $\delta>0$ is such that $\{v\in \RR^2: S+v\subseteq K\}\supseteq\{(a,0):|a|\leq \delta\}$. Let $p=(x_0,y_0)$ and $p'=(x_0',y_0')$ be points of $S$ and $K$ (respectively) with maximal second coordinates. Then either $y_0=y_0'$, or there is some $\epsilon>0$ such that $S+(0,\epsilon)\subseteq K$. Similarly, if $p''=(x_0'',y_0'')$ and $p'''=(x_0''',y_0''')$ are points of $S$ and $K$ (respectively) with minimal second coordinates, then either $y_0''=y_0'''$, or there is some $\epsilon>0$ such that $S-(0,\epsilon)\subseteq K$
\end{lemma}
\begin{proof}
	We only prove the first claim, as the second one is similar. Certainly $y_0'\geq y_0$ as $S\subseteq K$. Let us assume that $y_0'>y_0$, we show that if $\epsilon>0$ is sufficiently small then for any $q=(x_1,y_1)\in S$ we have $q+(0,\epsilon)\in K$. It is enough to consider the case $x_1\geq x_0'$.
	Let $L>0$ be such that $S\subseteq[-L,L]^2$. We know that $q'=(x_1+\delta,y_1)$  is in $K$. By convexity, the line segment between $q'$ and $p'$ also lies in $K$ and hence $\left(x_1, \frac{y_1-y_0'}{x_1+\delta-x_0'}(x_1-x_0')+y_0'\right)\in K$. But we have 
	\begin{align*}
	\left(\frac{y_1-y_0'}{x_1+\delta-x_0'}(x_1-x_0')+y_0'\right)-y_1=\frac{\delta}{x_1+\delta-x_0'}(y_0'-y_1)\geq \frac{\delta}{2L+\delta}(y_0'-y_0).
	\end{align*}
	It follows that $\epsilon=\frac{\delta}{2L+\delta}(y_0'-y_0)$ satisfies the conditions.
\end{proof}

The next lemma will be used to prove Hausdorff-continuity in the difficult case, i.e., when $I_S$ is one-dimensional.

\begin{lemma}\label{lemma_hausdorffclosecopy}
	Suppose that $K\subseteq\RR^2$ is compact and convex, and define the sets $I_S$ and $\AA_K$ as in Lemma~\ref{lemma_2dallhausdorff}. Assume that $u\in \RR^2$, $\delta>0$ and $S\in \AA_K$ such that $I_S$ has empty interior but $I_S\supseteq\{u+(a,0):|a|\leq \delta\}$. Then for all $\epsilon>0$ there is some $\eta>0$ such that whenever $S'\in \AA_K$ satisfies $d(S,S')<\eta$ then there is some $w\in I_{S'}$ with $|w-u|<\epsilon$.
\end{lemma}
\begin{proof}
We may assume $u=0$ (by replacing $K$ by $K-u$). Since $I_S$ has empty interior, by Lemma~\ref{lemma_2dsameprojections} we have $y_0=y_0'$ and $y_0''=y_0'''$ (using the notation in the statement of that lemma). If $S'\in \AA_K$ and $d(S,S')<\eta$, we know $S'+(0,z)\subseteq \RR\times [y_0'',y_0]$ for some $z\in \RR$ with $|z|<\eta$. (Indeed, we can pick $z$ such that the largest second coordinate of a point in $S'$ is $y_0+z$.) We show that if $\eta$ is small enough, then we must have $(0,z)\in I_{S'}$. (Then we are done, as we can choose $\eta<\epsilon$.) By replacing $S'$ with $S'+(0,z)$ (and $\eta$ by $2\eta$), we may assume that $z=0$.

So we need to show that for any point $q=(x_1,y_1)\in S'$, we have $q\in K$ (if $\eta$ is small). We know there is some $q'=(x_2,y_2)\in S$ with $|x_1-x_2|<\eta$, $|y_1-y_2|<\eta$. We may assume that $y_1\geq y_2$. We wish to show that for some  $s\in[-\delta,\delta]$, $q-(s,0)$ must lie on the line segment between $p=(x_0,y_0)$ and $q'=(x_2,y_2)$. (Then we are done, since $p,q'\in S$, $S+(s,0)\subseteq K$, and $K$ is convex.)
\begin{figure}[h!]
	\includegraphics[clip,trim=0cm 0cm 0cm 0cm, width=0.6\linewidth]{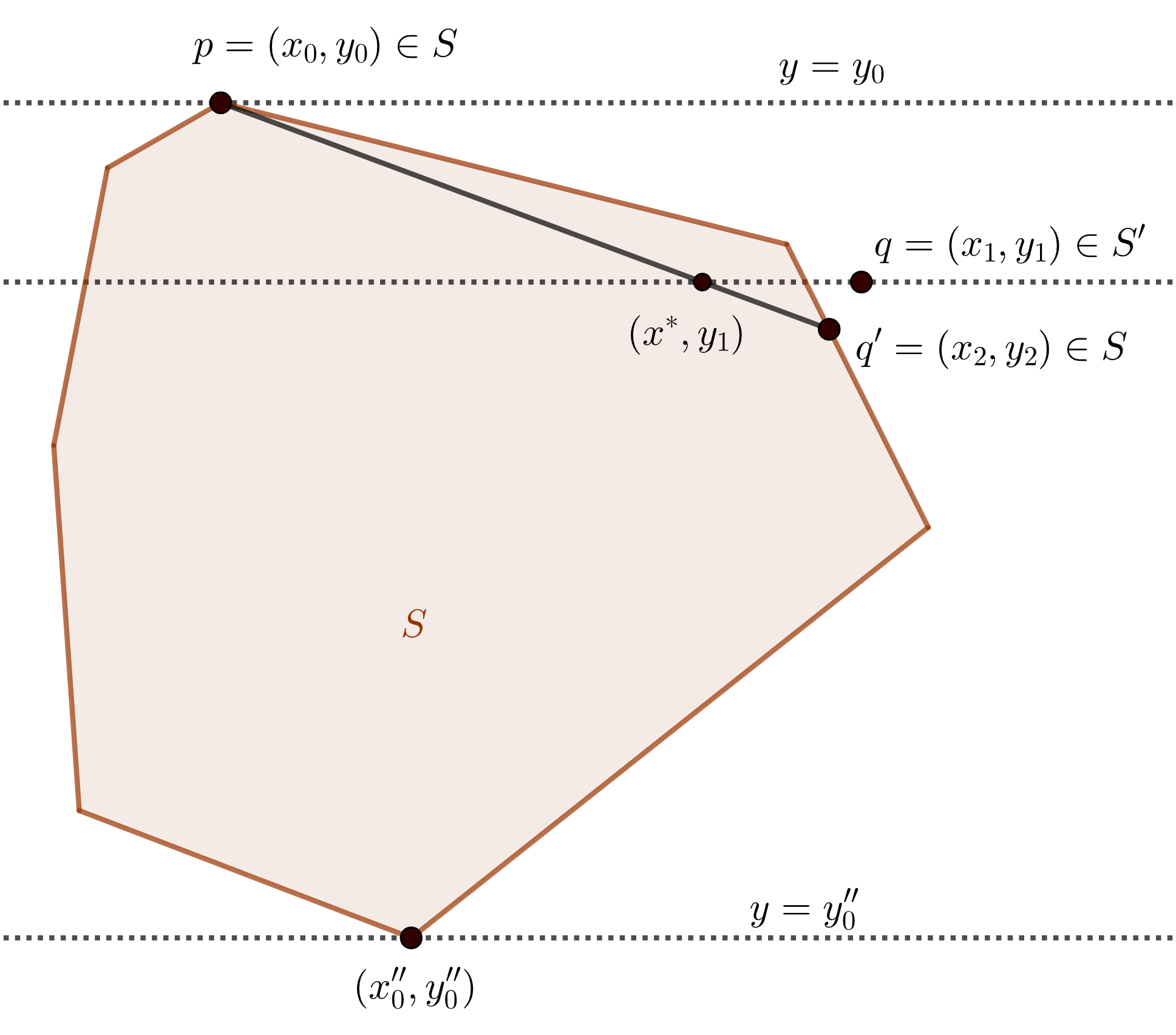}
	\centering
	\captionsetup{justification=centering}
	\caption{The points used in the proof of Lemma~\ref{lemma_hausdorffclosecopy}.}
	\label{figure_2dcalc}
\end{figure}

First assume that $y_2=y_0$. So $y_0=y_1=y_2$. But then $q=q'+(s,0)$ for some $s\in (-\eta,\eta)$, so our claim follows easily by picking $\eta<\delta$.

So let us now assume $y_2\not =y_0$ (so $y_2<y_0$). Observe that points $(x,y)$ on the segment between $p$ and $q'$ are the ones satisfying the equation $x-x_0=\frac{x_0-x_2}{y_0-y_2}(y-y_0)$ and have $y_2\leq y\leq y_0$. It follows that $(x^*,y_1)$ is on this segment, where $x^*=x_0+\frac{x_0-x_2}{y_0-y_2}(y_1-y_0)$. We have
\begin{align*}
|x^*-x_2|&=\left|x_0+\frac{x_0-x_2}{y_0-y_2}(y_1-y_0)-x_2\right|\\
&=\left|\frac{x_0-x_2}{y_0-y_2}(y_1-y_2)\right|.
\end{align*}
We will use the following claim to bound this quantity.

\textbf{Claim.} There is some $\mu>0$ depending on $S, \delta$ only such that whenever $(x,y)\in S$ and $y>y_0-\mu$, then there is some $\bar{x}_0$ such that $(\bar{x}_0,y_0)\in S$ and $|\bar{x}_0-x|<\delta/2$.

\textbf{Proof of Claim.} If this is not true, then for all $n$ we can find $(x(n),y(n))\in S$ such that $y(n)>y_0-1/n$ and whenever $(\bar{x}_0,y_0)\in S$ then $|\bar{x}_0-x(n)|\geq \delta/2$. By taking a subsequence, we may assume that $(x(n),y(n))$ converges to some $(\tilde{x},\tilde{y})\in S$. But then $\tilde{y}=y_0$ and $\tilde{x}-x(n)\to 0$, giving a contradiction and proving the claim.\qed

By the claim above, we can modify $x_0$ if necessary so that either $y_0-y_2\geq \mu$ or $|x_0-x_2|<\delta/2$.

In the first case we get $|x^*-x_2|\leq \frac{|x_0-x_2|}{\mu}|y_1-y_2|$. Let $L>0$ be such that $S\subseteq [-L,L]^2$, then we get $|x^*-x_2|\leq \frac{2L}{\mu}\eta$ and hence $|x^*-x_1|\leq \eta+\frac{2L}{\mu}\eta$. This converges to $0$ (independently of $q,q'$) as $\eta\to 0^+$, as required.

On the other hand, if $|x_0-x_2|<\delta/2$ then, using $y_0-y_2\geq y_0-y_1$, we get $|x^*-x_2|\leq \delta/2$ and hence $|x^*-x_1|\leq \delta/2+\eta$, which is less than $\delta$ for $\eta<\delta/2$.
\end{proof}

\begin{proof}[Proof of Lemma~\ref{lemma_2dallhausdorff}]
	First note that all sets of the form $I_S$ are convex and compact. Let $S\in\AA_K$ be arbitrary, we show $\psi$ is continuous at $S$, i.e., whenever $S_n\to S$ with $S_n\in\AA_K$, then $d(I_{S_n},I_S)\to 0$. First we show $\max_{x\in I_{S_n}}d(x,I_S)\to 0$. Indeed, if this is not true, then by taking an appropriate subsequence $(S_{k(n)})$ we get that there is a sequence $(x_{n})$ with $x_n\in I_{S_{k(n)}}$ such that $d(x_n,I_S)\not\to 0$ and $x_n\to x$ for some $x$. But we have $S_{k(n)}+x_n\subseteq K$ for all $n$. Hence $S+x\subseteq K$, i.e., $x\in I_S$. (Indeed, for any $s\in S$ we can take a sequence $(s_n)$ with $s_n\in S_{k(n)}$ and $(s_n)\to s$. Then $s_n+x_n\in K$ for all $n$, so, by taking limits, $s+x\in K$.) But then $d(x_n,I_S)\to 0$, giving a contradiction. So $\max_{x\in I_{S_n}}d(x,I_S)\to 0$.

It remains to show that $\max_{x\in I_{S}}d(x,I_{S_n})\to 0$. Observe that it suffices to show that $d(x,I_{S_n})\to 0$ for any point $x\in I_S$. Indeed, the functions $x\mapsto d(x,I_{S_n})$ are $1$-Lipschitz on the compact domain $I_S$, so pointwise convergence implies uniform convergence. We consider three cases: when $I_S$ is a single point, when $I_S$ is one-dimensional, i.e., $I_S=\{(1-t)a+tb: t\in [0,1]\}$ for some $a,b\in \RR^2$ distinct, and when $I_S$ is two-dimensional, i.e., has non-empty interior.

First assume that $I_S=\{p\}$ is a single point. Then trivially $$d(p,I_{S_n})=\min_{x\in I_{S_n}}d(p,x)\leq \max_{x\in I_{S_n}}d(p,x)=\max_{x\in I_{S_n}}d(x,I_S)\to 0,$$
giving the claim.

Next, assume that $I_S$ is one-dimensional (i.e., a segment). By taking an appropriate rotation and translation, we may assume that $I_S=[-\delta,\delta]\times\{0\}$ for some $\delta>0$. Let $x\in I_S$ and $\epsilon>0$ be arbitrary, we show $d(x,I_{S_n})<\epsilon$ for $n$ large enough. We may assume that $\epsilon<\delta$. Let $x'\in [-\delta+\epsilon/2,\delta-\epsilon/2]\times\{0\}$ be such that $|x-x'|\leq\epsilon/2$. Since $I_S\supseteq \{x'+(a,0):|a|\leq \delta/2\}$, Lemma~\ref{lemma_hausdorffclosecopy} shows that for all $n$ large enough there is some $w\in I_{S_n}$ with $|w-x'|\leq \epsilon/4$. But then we also have $|w-x|<\epsilon$, as required.

Finally, assume that $I_S$ is two-dimensional, i.e., has non-empty interior. Let $x\in I_S$ and $\epsilon>0$ be arbitrary, we show $d(x,I_{S_n})<\epsilon$ for $n$ large enough. We can find $x'\in I_{S}$ with $|x'-x|<\epsilon$ such that $x'$ is in the interior of $I_{S}$, i.e., $I_S$ contains a ball of radius $r>0$ around $x'$. Then whenever $d(S',S)<r$, we have $x'\in I_{S'}$. (Indeed, $x'+S'\subseteq x'+B_r(0)+S\subseteq I_S+S\subseteq K$, where $B_r(0)$ denotes the ball of radius $r$ centred at $0$.) Hence $x'\in I_{S_n}$ for $n$ large enough, giving the claim.
\end{proof}

\begin{proof}[Proof of Theorem~\ref{theorem_continuousin2d}]
	Let $S'$ be the the closure of $S$. Then $\{v\in \RR^2: \rho(S)+v\subseteq K\}=\{v\in \RR^2: \rho(S')+v\subseteq K\}$ for all $\rho\in \SO(2)$. By replacing $S$ by $S'$, we may assume that $S$ is compact. Then the result follows easily from Lemma~\ref{lemma_2dallhausdorff} and Lemma~\ref{lemma_chebyshevproperties} by letting $f(\rho)$ be the Chebyshev centre of $I_{\rho(S)}$.
\end{proof}

We finish this section by proving Theorem~\ref{theorem_continuouscounterexample}. Informally, the construction can be described as follows. Take a circle of diameter $1$ in the $xy$ plane, and start moving it in the $x$ direction while simultaneously rotating it around the $x$ axis. Stop when the rotated circle gets back to the $xy$ plane, and take the convex hull of the points traversed. See Figure~\ref{figure_continuouscounterexample}. The discontinuity will come at the direction $(0,1,0)$ by considering directions of the form $(0,y,\pm\sqrt{1-y^2})$, $y\to 1^-$. The formal proof is given below.
\begin{figure}[h!]
	\centering
	\begin{subfigure}[h!]{0.45\textwidth}
		\centering
		\includegraphics[width=\textwidth]{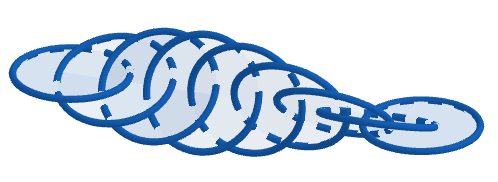}
					\captionsetup{justification=centering}
		\caption{Some phases of the circle being rotated and translated.}
		\label{figure_circlephases}
	\end{subfigure}
\hfill
	\begin{subfigure}[h!]{0.45\textwidth}
		\centering
		\includegraphics[width=\textwidth]{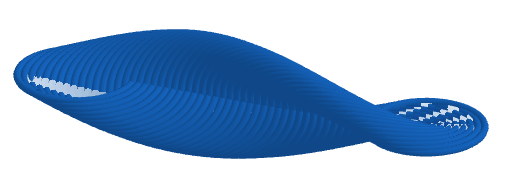}
			\captionsetup{justification=centering}
		\caption{The set of points traversed during the motion. The final construction is obtained by taking the convex hull of this set.}
		\label{figure_finalpicture}
	\end{subfigure}
	\captionsetup{justification=centering}
	\caption{The counterexample in Theorem~\ref{theorem_continuouscounterexample} is obtained by simultaneously translating and rotating a circle, and then taking convex hull of the points traversed.}
	\label{figure_continuouscounterexample}
\end{figure}
\begin{proof}[Proof of Theorem~\ref{theorem_continuouscounterexample}]
	Define the function $f: [0,\pi]\times \Sph^1\to \RR^3$ by letting
	$$f(t,x,y)=\frac{1}{2}(t+x,y\cos t, y\sin t).$$
	Let $K_0$ be the image of $f$ and let $K$ be the convex hull of $K_0$. Observe that $f$ is continuous and the domain of $f$ is compact, hence $K_0$ is compact. It follows that $K$ is convex and compact. Also, note that if $v\in \Sph^2$, then $v$ can be written as $v=(r_1,r_2\cos \varphi, r_2\sin\varphi)$ for some $r_1,r_2\in\RR$ with $r_1^2+r_2^2=1$ and $\varphi\in [0,\pi]$. Then $f(\varphi, r_1,r_2)-f(\varphi, -r_1,-r_2)=(r_1,r_2\cos \varphi,r_2\sin \varphi)=v$, so $I_v$ is non-empty, where $I_v=\{u\in\RR^2:u,u+v\in K\}$. It remains to show that there is no continuous function $\psi: \Sph^2\to K$ such that $\psi(v)\in I_v$ for all $v$.
	
	Let $C=\{(a,b,c)\in\RR^3:b^2+c^2=1/4\}$ and $C'=\{(a,b,c)\in\RR^3:b^2+c^2\leq 1/4\}$. Observe that $K_0\subseteq C'$ and $$K_0\cap C=\left\{\frac{1}{2}(t,s\cos t,s\sin t):s=\pm 1, t\in[0,\pi]\right\}.$$ It is easy to deduce that $K\subseteq C'$ and $$K\cap C=\left\{\frac{1}{2}(t,s\cos t,s\sin t):s=\pm 1, t\in[0,\pi]\right\}\cup\left\{\frac{1}{2}(a,\pm 1,0):a\in [0,\pi]\right\}.$$ It is easy to deduce that if $v=(0,\cos \varphi, \sin\varphi)$ for some $\varphi\in(0,\pi)$, then $I_v$ consists of the single point $\frac{1}{2}(\varphi,-\cos\varphi,-\sin\varphi)$, and if $v=(0,\cos \varphi, \sin\varphi)$ for some $\varphi\in(-\pi,0)$, then $I_v$ consists of the single point $\frac{1}{2}(\pi+\varphi,\cos(\pi+\varphi),\sin(\pi+\varphi))=\frac{1}{2}(\pi+\varphi,-\cos\varphi,-\sin\varphi)$. It follows that if $\psi: \Sph^2\to K$ such that $\psi(v)\in I_v$ for all $v$, then $\psi$ cannot be continuous at $(0,1,0)$.
\end{proof}

\section{Segments in \texorpdfstring{$\RR^d$}{R^d}}\label{section_Kakeyareachable}
\subsection{Proof outline and some simple results}
Our goal in this section is to prove Theorem~\ref{theorem_Kakeyasegments} about Kakeya sets in $\RR^d$. Throughout this section, we assume that $d\geq 3$ and $K$ is a compact convex set in $\RR^d$ such that for all $v\in \Sph^{d-1}$, the set $I_v=\{u\in\RR^d: u, u+v\in K\}$ is non-empty. Note that $I_v$ is a compact convex set for all $v$.

Given $v, v'\in \Sph^{d-1}$, $u\in I_v, u'\in I_{v'}$ and $\gamma: [0,1]\to \Sph^{d-1}$ continuous with $\gamma(0)=v, \gamma(1)=v'$, say that $(v',u')$ is reachable from $(v,u)$ along $\gamma$ if there is a continuous $\delta:[0,1]\to K$ such that $\delta(t)\in I_{\gamma(t)}$ for all $t$, $\delta(0)=u$ and $\delta(1)=u'$. We say that $v'$ is reachable from $v$ along $\gamma$ if there exist $u,u'$ such that $(v',u')$ is reachable from $(v,u)$ along $\gamma$, and we say $v'$ (or $(v',u')$) is reachable from $v$ (respectively, $(v,u)$) if there exists a $\gamma$ along which it is reachable. Given a subset $X\subseteq \Sph^{d-1}$, $\epsilon\geq 0$ and $\gamma: [0,1]\to \Sph^{d-1}$ we say that $\gamma$ is $\epsilon$-close to $X$ if for all $t\in [0,1]$ there is some $p\in X$ such that $|p-\gamma(t)|\leq\epsilon$. Given $\epsilon\geq 0$ and $\gamma,\gamma': [0,1]\to \Sph^{d-1}$ we say that $\gamma$ is $\epsilon$-close to $\gamma'$ if it is $\epsilon$-close to the image of $\gamma'$. (Note that this relation is not symmetric.)

So, using this terminology, our goal is to prove the following result.

\begin{theorem}\label{theorem_reachable}
	Let $v,v'\in \Sph^{d-1}$ and $u\in I_v, u'\in I_{v'}$, and let $\gamma:[0,1]\to \Sph^{d-1}$ be continuous such that $\gamma(0)=v, \gamma(1)=v'$. Then for any $\epsilon>0$, $(v',u')$ is reachable from $(v,u)$ along a path which is $\epsilon$-close to $\gamma$.
\end{theorem}

Note that the counterexample in Theorem~\ref{theorem_continuouscounterexample} shows that it is not necessarily true that $v'$ is reachable from $v$ along $\gamma$ (or along a path $0$-close to $\gamma$).\medskip

We now briefly discuss our approach to proving Theorem~\ref{theorem_reachable}. It is easy to see that if $p\in \Sph^{d-1}$ is such that $I_p$ has non-empty interior, then every $p'$ in some neighbourhood of $p$ is reachable from $p$. Furthermore, it is not difficult to deal with points $p$ such that $I_p$ is a single point. This means that the complicated case is when $I_p$ is not a single point, but has empty interior (i.e., its dimension is between $1$ and $d-1$). We will prove (Lemma~\ref{lemma_spheresaroundtypeline}) that in the neighbourhood of such points $p$, there are `many' points $q$ with $I_q$ having non-empty interior. Moreover, we will show that if for such a $p$ we start moving on the sphere $\Sph^{d-1}$ from $p$ in some direction, then for `most' directions we initially only encounter points $q$ such that $I_q$ has non-empty interior, and that these $q$ are reachable from $p$. We will deduce (Lemma~\ref{lemma_pathsavoidingbadcase}) that Theorem~\ref{theorem_reachable} holds for $\gamma$ if for all points $v$ on $\gamma$ such that $I_v$ has empty interior and is not a single point, the tangent to $\gamma$ at $v$ is not in some special set of `forbidden' directions. Finally, we will show that we can perturb $\gamma$ slightly to make sure that we avoid such cases. We note that in some sense we can have `many' points $p\in \Sph^{d-1}$ such that $I_p$ is not a single point but has empty interior. For example, if $K=\{(x,y,z)\in \RR^3: x\in[-1,1],y^2+z^2\leq 1/4\}$, then all $p$ along a great circle have this property.

We believe the reader will not lose much by focusing on the case $d=3$: some of the lemmas are easier to visualise and prove in that case, but the main ideas of the proof are the same.\medskip

Let us start with some simple observations.

\begin{lemma}\label{lemma_canchooseendpoints}
	Suppose that $v,v'\in \Sph^{d-1}$ and $\gamma:[0,1]\to \Sph^{d-1}$ are such that $v'$ is reachable from $v$ along $\gamma$. Let $u\in I_v$, $u'\in I_{v'}$ be arbitrary. Then $(v',u')$ is reachable from $(v,u)$ along a path which has the same image as $\gamma$ (and hence is $0$-close to $\gamma$). 
\end{lemma}
\begin{proof}
	Let $w\in I_v$, $w'\in I_{v'}$ and $\delta: [0,1]\to K$ be such that $\delta$ is continuous, $\delta(t)\in I_{\gamma(t)}$ for all $t$, $\delta(0)=w$ and $\delta(1)=w'$. Define $\gamma'$ and $\delta'$ by setting
	\begin{equation*}
	\gamma'(t)=
	\begin{cases*}
	v & if $t\in[0,1/3]$\\
	\gamma(3(t-1/3)) & if $t\in[1/3,2/3]$\\
	v' & if $t\in [2/3,1]$
	\end{cases*}
	\end{equation*}
	and
	\begin{equation*}
	\delta'(t)=
	\begin{cases*}
	(1-3t)u+3tw & if $t\in[0,1/3]$\\
	\delta(3(t-1/3)) & if $t\in[1/3,2/3]$\\
	(1-3(t-2/3))w'+3(t-2/3)u' & if $t\in [2/3,1].$
	\end{cases*}
	\end{equation*}
	The statement of the lemma follows easily, using that $I_v, I_{v'}$ are convex.
\end{proof}
Note that Lemma~\ref{lemma_canchooseendpoints} implies that if $v'$ is reachable from $v$ (along some path which is $\epsilon$-close to $X$) and $v''$ is reachable from $v'$ (along some path which is $\epsilon$-close to $Y$) then $v''$ is reachable from $v$ (along some path which is $\epsilon$-close to $X\cup Y$).
\begin{lemma}\label{lemma_nicetype}
	Assume that $V\subseteq \Sph^{d-1}$ is such that for all $v\in V$, $I_v$ has non-empty interior. Assume furthermore that $\gamma: [0,1]\to V$ is continuous. Then $\gamma(1)$ is reachable from $\gamma(0)$ along a path which is $0$-close to $\gamma$.
\end{lemma}
\begin{proof}
	For all $t\in [0,1]$ we can find some $r_t>0$, $p_t\in K$ such that $I_{\gamma(t)}$ contains an open ball of radius $r_t$ around $p_t$, i.e., whenever $|z|<r_t$ then $p_t+z, p_t+z+\gamma(t)\in K$. It follows that whenever $|\gamma(s)-\gamma(t)|<r_t$ then $p_t, p_t+\gamma(s)\in K$, i.e., $p_t\in I_{\gamma(s)}$. Let $\eta_t>0$ be such that $|\gamma(s)-\gamma(t)|<r_t$ whenever $|s-t|<\eta_t$. By compactness of $[0,1]$, we can find 
	some $r>0$ such that whenever $s\in[0,1]$ then there is some $t_s\in [0,1]$ such that $|s-t_s|\leq \eta_{t_s}-r$. Pick some $N>1/r$ integer, and let $x(i)=i/N$ ($i=0,\dots,N$). Then $\gamma(x({i+1}))$ is reachable from $\gamma(x(i))$ along a path which has the same image is $\gamma|_{[x(i),x({i+1})]}$ (the corresponding function $\delta$ is constant $p_{t_{x(i)}}$). Using Lemma~\ref{lemma_canchooseendpoints} several times,  and concatenating the appropriate paths, we get that $\gamma(1)$ is reachable from $\gamma(0)$ along a path with the same image as $\gamma$.
\end{proof}
In light of Lemma~\ref{lemma_nicetype}, finding points $v$ such that $I_v$ has non-empty interior is useful for proving reachability. The next lemma gives a convenient condition for checking that $I_v$ has non-empty interior.
\begin{lemma}\label{lemma_longsegment}
	If $v\in \Sph^{d-1}$ and there is some $\lambda>1$ and $u\in K$ such that $u+\lambda v\in K$, then $I_v$ has non-empty interior.
\end{lemma}
\begin{proof}
	We may assume that $u=0$. If $p\in K$, then $(1-1/\lambda)p\in K$ and $(1-1/\lambda)p+(1/\lambda) \lambda v\in K$ by convexity, so $(1-1/\lambda)p\in I_{v}$. Given some $w\in \Sph^{d-1}$, there are points $p_1,p_2\in K$ such that $p_2-p_1=w$. Then $(1-1/\lambda)p_i\in I_{v}$ for $i=1,2$, and therefore $I_v$ contains two points $q_1, q_2$ with $q_2-q_1=(1-1/\lambda)w$. So we can pick $e_1, f_1, \dots, e_d, f_d\in I_v$ such that $f_i-e_i$ is the vector with all coordinates zero, except the $i$th coordinate, which is $1-1/\lambda$. Let $c=\frac{1}{2d}\sum(e_i+f_i)$. By convexity of $I_v$, it is easy to see that $c\in I_v$, and whenever $|x_i|\leq \frac{1}{2d}(1-1/\lambda)$ for all $i$ then $c+(x_1,\dots,x_d)\in I_v$. So $I_v$ contains a ball of radius $\frac{1}{2d}(1-1/\lambda)$ around $c$.
\end{proof}
The following useful lemma gives another condition for finding $v$ such that $I_v$ has non-empty interior, and it also gives some restrictions on what $I_v$ can look like when $I_v$ has empty interior: $I_v-I_v$ must be perpendicular to $v$.
\begin{lemma}\label{lemma_largeinnerproducts}
	Suppose that $p\in \Sph^{d-1}$, $x,q\in \RR^d$ such that $\langle p,q\rangle>1$ and $x, x+q\in K$. Then $I_p$ has non-empty interior.
	
	In particular, if $v\in \Sph^{d-1}$ and $u,w\in \RR^d$ such that $\langle v,w\rangle \not =0$ and $u,u+w\in I_v$, then $I_v$ has non-empty interior.
\end{lemma}
\begin{proof}
	Let $\epsilon\in (0,1)$ be small enough so that $0\not =|p-\epsilon q|<1-\epsilon$. Note that such an $\epsilon$ exists, since $|p-\epsilon q|^2=1-2\langle p,q\rangle\epsilon+|q|^2\epsilon^2$ is less than $1-2\epsilon+\epsilon^2$ for $\epsilon$ small enough, as $\langle p,q\rangle>1$. Let $p'=\frac{p-\epsilon q}{|p-\epsilon q|}$. Note that $|p'|=1$. We know that there is some $y\in \RR^d$ such that $y,y+p'\in K$. Let
	\begin{align*}
	z&=\frac{\epsilon}{\epsilon+|p-\epsilon q|} x+\frac{|p-\epsilon q|}{\epsilon+|p-\epsilon q|}y,\\
	z'&=\frac{\epsilon}{\epsilon+|p-\epsilon q|} (x+q)+\frac{|p-\epsilon q|}{\epsilon+|p-\epsilon q|}(y+p').\\
	\end{align*}
	Then $z,z'\in K$ by convexity. But
	\begin{align*}z'-z&=\frac{\epsilon}{\epsilon+|p-\epsilon q|} q+\frac{|p-\epsilon q|}{\epsilon+|p-\epsilon q|}p'\\
	&=\frac{1}{\epsilon+|p-\epsilon q|}p
	\end{align*}
	But $\frac{1}{\epsilon+|p-\epsilon q|}>1$, so $I_p$ has non-empty interior by Lemma \ref{lemma_longsegment}.
	
	For the final part of the lemma, we may assume $\langle v,w\rangle >0$ (otherwise replace $u$ by $u+w$ and $w$ by $-w$). But then $u, u+v+w\in K$, so we can apply the first part of the lemma with $p=v$, $q=v+w$, $x=u$.
\end{proof}

\subsection{The main lemmas}
The following lemma is one of the key observations. Essentially, it says that if $I_v$ has empty interior but is not a single point, then for `most' points $p$ around $v$ the set $I_p$ has non-empty interior, and those $p$ are reachable from $v$.
\begin{lemma}\label{lemma_spheresaroundtypeline}
	Suppose that $v\in \Sph^{d-1}$, and $u,w\in \RR^d$ such that $u
	,u+w\in I_v$, $0<|w|<1$ and $\langle v,w\rangle=0$. Let $P=\{p\in \Sph^{d-1}: \langle p, v+w \rangle > 1 \}\cup \{p\in \Sph^{d-1}: \langle p, v-w \rangle > 1 \}$. Then $I_p$ has non-empty interior for all $p\in P$. Moreover, whenever $p\in P$, then $p$ is reachable from $v$ along a path which is $2|p-v|$-close to $\{v\}$.
\end{lemma}
Note that the condition $\langle v,w\rangle=0$ holds automatically when $I_v$ has empty interior by the final part of Lemma~\ref{lemma_largeinnerproducts}. Figure~\ref{figure_Kakeyaspheres} shows the set $P$ in the case $d=3$.
\begin{figure}[h!]
	\includegraphics[clip,trim=0cm 0cm 0cm 0cm, width=0.3\linewidth]{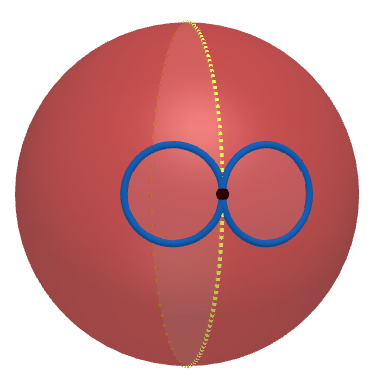}
	\centering
	\captionsetup{justification=centering}
	\caption{The set $P$ in Lemma~\ref{lemma_spheresaroundtypeline} is the region enclosed by the two blue circles ($d=3$). The point $v$ is the intersection of the two circles, and $w$ is parallel to the line connecting the centres of the blue circles. The yellow dotted great circle gives the only direction (for $d=3$) not pointing to the inside of the two circles.}
	\label{figure_Kakeyaspheres}
\end{figure}
\begin{proof}
	The claim that $I_p$ has non-empty interior for all $p\in P$ follows directly from Lemma~\ref{lemma_largeinnerproducts}. For the second claim, let $p\in P$ be arbitrary. We may assume $\langle p,v+w\rangle>1$ (otherwise replace $u$ by $u+w$ and $w$ by $-w$). Note that $\langle v,p\rangle=\langle v+w,p\rangle-\langle w,p\rangle>1-1=0$, and similarly $\langle w,p\rangle>0$. Pick some small $\lambda\in (0,1)$ (to be specified later). It is easy to see that if we write $\gamma(t)=\frac{v+2t\lambda w}{|v+2t\lambda w|}$ and $\delta(t)=u$ for all $t\in [0,1/2]$, then $\delta(t)\in I_{\gamma(t)}$ for all $t\in [0,1/2]$. Furthermore, if we write $q(s)=\frac{(1-s)(v+\lambda w)+sp}{|(1-s)(v+\lambda w)+sp|}$ for all $s\in [0,1]$, then $\langle q(s), v+w\rangle>1$ for all $s\in [0,1]$. Indeed, it is easy to check that $\langle (1-s)(v+\lambda w)+sp, v+w\rangle> 0$, and we have
	\begin{align*}
	|(1-s)(v+\lambda w)+sp|^2 &=(1-s)^2(1+\lambda^2|w|^2)+s^2+2s(1-s)\langle p,v+\lambda w\rangle\\
	&\leq (1-s)^2(1+\lambda^2|w|^2)+s^2+2s(1-s)\langle p, v+w\rangle
	\end{align*}
	and	
	\begin{align*}
	\langle (1-s)(v+\lambda w)&+sp,v+w\rangle^2 =\left((1-s)(1+\lambda|w|^2)+s\langle p,v+w\rangle\right)^2\\
	&=(1-s)^2(1+\lambda |w|^2)^2+s^2\langle p,v+w\rangle^2+2s(1-s)(1+\lambda |w|^2)\langle p,v+w\rangle\\
	&> (1-s)^2(1+\lambda|w|^2)+s^2+2s(1-s)\langle p,v+w\rangle\\
	&\geq (1-s)^2(1+\lambda^2|w|^2)+s^2+2s(1-s)\langle p, v+w\rangle,
	\end{align*}
	giving $\langle q(s), v+w\rangle^2>1$.
	
	So $I_{q(s)}$ has non-empty interior for all $s$. Using Lemma~\ref{lemma_nicetype} (and Lemma~\ref{lemma_canchooseendpoints}), it is easy to deduce that we can extend $\gamma, \delta$ to $[0,1]$ such that $\delta(t)\in I_{\gamma(t)}$ for all $t$ and for all $t\geq 1/2$ there is some $s\in [0,1]$ such that $\gamma(t)=q(s)$.
	
	Now, if $t\leq 1/2$, then
	\begin{align*}
	|v-\gamma(t)|^2&=\left|v-\frac{v+2t\lambda w}{|v+2t\lambda w|}\right|^2\\
	&=2-2\left\langle v,\frac{v+2t\lambda w}{|v+2t\lambda w|}\right\rangle\\
	&=2-\frac{2}{|v+2t\lambda w|}\\
	&=2-\frac{2}{(1+(2t\lambda|w|)^2)^{1/2}}\\
	&\leq 2-\frac{2}{(1+\lambda^2|w|^2)^{1/2}}.
	\end{align*}
	
	Furthermore, if $t\geq 1/2$ and $\gamma(t)=q(s)$, then
	\begin{align*}
	|v-\gamma(t)|^2&=\left|v- \frac{(1-s)(v+\lambda w)+sp}{|(1-s)(v+\lambda w)+sp|}\right|^2\\
	&=2-2\left\langle v, \frac{(1-s)(v+\lambda w)+sp}{|(1-s)(v+\lambda w)+sp|}\right\rangle\\
	&=2-2\frac{(1-s)+s\langle v,p\rangle}{|(1-s)(v+\lambda w)+sp|}\\
	&\leq 2-2\frac{\langle v,p\rangle}{|(1-s)(v+\lambda w)+sp|}\\
	&\leq 2-2\frac{\langle v,p\rangle}{\max\{|v+\lambda w|,|p|\}}\\
	&=2-\frac{2\langle v,p\rangle}{(1+\lambda^2|w|^2)^{1/2}}.
	\end{align*}
	It follows that $|v-\gamma(t)|\leq \left(2-\frac{2\langle v,p\rangle}{(1+\lambda^2|w|^2)^{1/2}}\right)^{1/2}$ for all $t$. But $\lambda\in (0,1)$ was arbitrary, and taking $\lambda\to 0^+$ we have $\left(2-\frac{2\langle v,p\rangle}{(1+\lambda^2|w|^2)^{1/2}}\right)^{1/2}\to (2-2\langle v,p\rangle )^{1/2}=|v-p|$. It follows that we can choose $\lambda$ such that $|v-\gamma(t)|\leq 2|v-p|$ for all $t$.
\end{proof}

The next lemma is one of the main corollaries of Lemma~\ref{lemma_spheresaroundtypeline}.

\begin{lemma}\label{lemma_pathsavoidingbadcase}
	Let $\epsilon>0$, and suppose that $\gamma:[0,1]\to \Sph^{d-1}$ is continuously differentiable such that for all $t\in [0,1]$, one of the following holds.
	\begin{enumerate}
		\item $I_{\gamma(t)}$ has non-empty interior;
		\item $I_{\gamma(t)}$ has empty interior, but there exist $u,w\in \RR^d$ such that $u,u+w\in I_{\gamma(t)}$
		and $\langle w,\gamma'(t)\rangle\not =0$;
		\item $I_{\gamma(t)}$ is a single point.
	\end{enumerate}
	Then $\gamma(1)$ is reachable from $\gamma(0)$ along a path which is $\epsilon$-close to $\gamma$.
\end{lemma}
Note that in the second case we must have $\langle \gamma(t),w\rangle=0$ by the final part of Lemma~\ref{lemma_largeinnerproducts}.
\begin{proof}
	Let $T_i$ be the set of $t\in [0,1]$ belonging to the $i$th case above $(i=1,2,3)$. We first claim that $T_3$ is closed. Indeed, it is easy to see that $T_1$ is open, and if $t\in T_2$ then by Lemma~\ref{lemma_spheresaroundtypeline} there is some $\epsilon>0$ such that $((t-\epsilon,t+\epsilon)\setminus\{t\})\cap [0,1]\subseteq T_1$.
	
	Since $T_3$ is closed, $[0,1]\setminus T_3$ is a union of disjoint (open) intervals: $[0,1]\setminus T_3=\bigcup_{J\in \JJ}J$, where for all $J$ either $J=(a_J, a_J+b_J)$ (with $0\leq a_J<a_J+b_J\leq 1$), or $J=[0,b_J)$, or $J=(a_J,1]$, or $J=[0,1]$ (and $J\cap J'=\emptyset$ if $J\not =J'$). For each $J\in \JJ$ (and positive integer $m$), define $J_m$ as follows:
	\begin{itemize}
		\item If $J=(a_J, a_J+b_J)$, let $J_m=[a_J+\frac{1}{3m}b_J,a_J+(1-\frac{1}{3m})b_J]$;
		\item If $J=[0,b_J)$, let $J_m=[0,(1-\frac{1}{3m})b_J]$;
		\item If $J=(a_J,1]$, let $J_m=[a_J+\frac{1}{3m}(1-a_J),1]$;
		\item If $J=[0,1]$, let $J_m=[0,1]$.
	\end{itemize}
	Note that $\bigcup_{m\geq 1} J_m=J$ and $J_1\subseteq J_2\subseteq J_3\subseteq\dots$. Let us also write $J_0=\emptyset$ for all $J$.
	
	Observe that if $t\in T_1\cup T_2$, then for any $\eta>0$ there exists $\mu>0$ such that if $t'\in (t-\mu,t+\mu)\cap [0,1]$ 
	then $\gamma(t')$ is reachable from $\gamma(t)$ along a path which is $\eta$-close to $\{\gamma(t)\}$.
	Indeed, this is easy to see (and follows from Lemma~\ref{lemma_nicetype}) when $t\in T_1$, and follows from Lemma~\ref{lemma_spheresaroundtypeline} when $t\in T_2$.
	
	\textbf{Claim.} We can recursively construct $\alpha_m:\bigcup_{J\in \JJ}J_m\to \Sph^{d-1}$ and $\beta_m:\bigcup_{J\in \JJ}J_m\to \RR^d$ continuous functions such that
	\begin{enumerate}
		\item $\beta_m(t)\in I_{\alpha_m(t)}$ for all $t$ (when defined);
		\item If $m'>m$ then $\alpha_{m'}$ and $\beta_{m'}$ extend $\alpha_{m}$ and $\beta_m$, respectively;
		\item For all $J\in\JJ$ and $m>0$ we have $\alpha_m(\min J_m)=\gamma(\min J_m)$ and $\alpha_m(\max J_m)=\gamma(\max J_m)$;
		\item If $t\in J_{m}\setminus J_{m-1}$, then there is some $t'\in J_{m}$ such that  
		$|t-t'|< \operatorname{length}(J)/m$ and $|\alpha_{m}(t)-\gamma(t')|<\min\{\epsilon,\operatorname{length}(J)/m\}$.\label{item_close}
	\end{enumerate}

\textbf{Proof of Claim. }It is enough to show that whenever $a<b$, $[a,b]\subseteq J$, $u\in I_{\gamma(a)}, v\in I_{\gamma(b)}$ and $\eta>0$, then there exist $f:[a,b]\to \Sph^{d-1}$ and $g:[a,b]\to K$ continuous functions such that $g(t)\in I_{f(t)}$ for all $t$, $g(a)=u, g(b)=v, f(a)=\gamma(a),f(b)=\gamma(b)$, and for all $t$ there is some $t'\in [a,b]$ with $|t-t'|<\eta$ such that $|f(t)-\gamma(t')|\leq\eta$. For each $t$, pick $\mu_t$ as in the observation above, we may assume $\mu_t<\eta$ for all $t$. Using the compactness of $[a,b]$, if $N$ is large enough and we write $x(j)=a+j(b-a)/N$, then for all $j$ we have $x(j)-x(j-1)<\eta$ and there is some $t_j$ such that $|x(j)-t_j|,|x(j-1)-t_j|<\mu_{t_j}$. But then $\gamma(x(j)),\gamma(x(j+1))$ are both reachable from $\gamma(t_j)$ along a path which is $\eta$-close to $\{\gamma(t_j)\}$, and hence $x(j+1)$ is reachable from $x(j)$ along a path which is $\eta$-close to $\{\gamma(t_j)\}$. It follows that for any choice of $u_j\in I_{\gamma(x_j)}$ ($j=0,\dots,N$) there exist $f_j:[x(j),x(j+1)]\to \Sph^{d-1}$ and $g_j:[x(j),x(j+1)]\to K$ such that $g_j(t)\in I_{f_j(t)}$ for all $t$, $g_j(x(j))=u_j$, $g_j(x(j+1))=u_{j+1}$, $f_j(x(j))=\gamma(x(j))$, $f_j(x(j+1))=\gamma(x(j+1))$, and for all $t$ we have $|f_j(t)-\gamma(t_j)|\leq\eta$. Picking $u_0=u$ and $u_N=v$ and then putting together these $f_j, g_j$ gives the required functions $f,g$ and finishes the proof of the claim.
\qed\medskip

	Define $\alpha:[0,1]\to \Sph^{d-1}$ and $\beta:[0,1]\to \RR^d$ by setting $\alpha(t)$ to be $\alpha_m(t)$ and $\beta(t)$ to be $\beta_m(t)$ when $t\in T_1\cup T_2$ (and $m$ is large enough so that this exist), and when $t\in T_3$ then setting $\alpha(t)=\gamma(t)$ and $\beta(t)$ to be the unique point in $I_{\gamma(t)}$. It is clear that $\alpha(0)=\gamma(0)$, $\alpha(1)=\gamma(1)$, $\beta(t)\in I_{\alpha(t)}$ for all $t$, and $\alpha, \beta$ are continuous at all points in $T_1\cup T_2$. Also, $\alpha$ is $\epsilon$-close to $\gamma$. We show that $\alpha, \beta$ are continuous at all points in $T_3$ as well.
	
	We first prove that if $t\in T_3$ then $\alpha$ is continuous at $t$. Take any sequence $(t_n)\to t$ in $[0,1]$, we show $(\alpha(t_n))\to \alpha(t)=\gamma(t)$. If this is not true, then we can take a subsequence of $(\alpha(t_n))$ that converges to some $p\in \Sph^{d-1}$, $p\not =\gamma(t)$, so we may assume that $(\alpha(t_n))$ is convergent. Also, we may assume that $t_n\in T_1\cup T_2$ for all $n$ (since $\gamma$ is continuous, and $\alpha(t')=\gamma(t')$ if $t'\in T_3$). We may also assume that $(t_n)$ is either decreasing or increasing. Let $J(n)\in \JJ$ be such that $t_n\in J(n)$, and let $m(n)$ be the positive integer such that $t_n\in J(n)_{m(n)}\setminus J(n)_{m(n)-1}$. Furthermore, let $t'_n$ be as given by point \ref{item_close} above for $t_n\in J_{m(n)}\setminus J_{m(n)-1}$. Since $(t_n)$ is either increasing or decreasing, either $J(n)$ is eventually constant and $m(n)\to \infty$, or $J(n)$ takes infinitely many different values and $\operatorname{length}(J(n))\to 0$. In either case, we have $\operatorname{length}(J(n))/m(n)\to 0$. Hence $\alpha(t_n)-\gamma(t'_n)\to 0$ and $t_n-t_n'\to 0$. But then $t_n'\to t$ and hence $\gamma(t_n')\to \gamma(t)$, which implies $\alpha(t_n)\to \gamma(t)$, as claimed.
	
	We now show that $\beta$ is also continuous at all $t\in T_3$. Assume that $(t_n)$ is a sequence in $[0,1]$ converging to $t\in T_3$, we show $\beta(t_n)\to \beta(t)$. As before, by taking a subsequence we may assume that $\beta(t_{n})$ converges to some $p\in K$. But $\beta(t_{n})\in I_{\alpha(t_{n})}$ for all $n$, i.e., $\beta(t_{n}), \beta(t_{n})+\alpha(t_n)\in K$ for all $n$. Since $K$ is closed and $\alpha$ is continuous, by taking limits we get $p, p+\alpha(t)\in K$, i.e., $p\in I_{\alpha(t)}=I_{\gamma(t)}$. But $I_{\gamma(t)}=\{\beta(t)\}$, hence $p=\beta(t)$, as claimed.
\end{proof}

We will attempt to find a `good' path, i.e., one satisfying the conditions of Lemma~\ref{lemma_pathsavoidingbadcase}. Note that the only case we need to avoid is having a point $v$ on $\gamma$ such that $I_v$ has empty interior, is not a single point, and the tangent to $\gamma$ at $v$ is perpendicular to any $u-u'$ ($u,u'\in I_v$). To find such paths, it will be easier to work in $\RR^{d-1}$ instead of on $\Sph^{d-1}$, using that locally they have the same structure. The next lemma captures the key property coming from Lemma~\ref{lemma_spheresaroundtypeline} in terms of parametrizations.

While the formal statement is rather complicated, the lemma is intuitively quite simple, as we now explain. Let us focus on the case $d=3$. Using Figure~\ref{figure_Kakeyaspheres}, we know that if $\gamma$ is a path such that $\gamma(t)$ is a `bad point', i.e., the conditions of Lemma~\ref{lemma_pathsavoidingbadcase} are not satisfied there, then we get the two blue circles touching at $v=\gamma(t)$ such that no point in the regions enclosed by the circles can be a bad point for any path. Moreover, we also know that $\gamma$ must have tangent in the direction of the yellow dotted line at $t$. Our next lemma essentially states that if we take charts then we still get the blue circles whose interiors cannot contain bad points.

\begin{lemma}\label{lemma_takingchartsnew}
	Let $\varphi: \RR^{d-1}\to V$ be a smooth parametrization of some open set $V\subseteq \Sph^{d-1}$, and let $X\subseteq \RR^{d-2}$ be an open neighbourhood of $0$. Write $\gamma_x(t)=(t,x)$ for $x\in X, t\in [0,1]$ (so $\gamma_x:[0,1]\to \RR^{d-1}$). Let $Z$ be the set of all $(t,x)\in\RR^{d-1}$ ($t\in[0,1],x\in X$) such that for $v=(\varphi\circ\gamma_x)(t)=\varphi(t,x)$ the set $I_v$ has empty interior, but there exist $u,w\in \RR^d$ such that $u,u+w\in I_v$, $w\not =0$, $\langle v,w\rangle=0$ and $\langle w, (\varphi\circ\gamma_x)'(t)\rangle =0$. Let $X_Z=\{x\in X: (t,x)\in Z\textnormal{ for some }t\in [0,1] \}$, and assume that $x\in X_Z$ and $t_x\in [0,1]$ are such that $(t_x,x)\in Z$. Then there is some $w_x\in\RR^{d-2}$, $w_x\not =0$ such that the open ball of radius $|w_x|$ centred at $(t_x,x+w_x)$ is disjoint from $Z$.
\end{lemma}

The following lemma tells us that the conclusion of Lemma~\ref{lemma_takingchartsnew} guarantees that there are `few' points we need to avoid.

\begin{lemma}\label{lemma_Baire}
	Let $Z\subseteq \RR^{d-1}$ and let $X$ be an open neighbourhood of $0$ in $\RR^{d-2}$. Let $X_Z=\{x\in X: (t,x)\in Z\textnormal{ for some }t\in [0,1] \}$, and for each $x\in X_Z$ let $t_x\in [0,1]$ be arbitrary such that $(t_x, x)\in Z$. Assume that for each $x\in X_Z$ there is some $w_x\in\RR^{d-2}$, $w_x\not =0$ such that the open ball of radius $|w_x|$ centred at $(t_x,x+w_x)$ is disjoint from $Z$. Then $X_Z\not =X$.
\end{lemma}

Before we prove Lemma~\ref{lemma_takingchartsnew} and Lemma~\ref{lemma_Baire}, let us first put them together to obtain the lemmas we will use later.

\begin{lemma}\label{lemma_takingcharts}
	Let $\varphi: \RR^{d-1}\to V$ be a smooth parametrization of some open set $V\subseteq \Sph^{d-1}$, and let $X\subseteq \RR^{d-2}$ be an open neighbourhood of $0$. Write $\gamma_x(t)=(t,x)$ for $x\in X, t\in [0,1]$ (so $\gamma_x:[0,1]\to \RR^{d-1}$). Then there exists some $x\in X$ such that for all $\epsilon>0$, $\varphi(\gamma_x(1))$ is reachable from $\varphi(\gamma_x(0))$ along a path which is $\epsilon$-close to $\varphi\circ\gamma_x$.
\end{lemma}

\begin{proof}
	Define $Z$, $X_Z$ and $t_x$ (for $x\in X_Z$) as in Lemma~\ref{lemma_takingchartsnew}. By Lemma~\ref{lemma_takingchartsnew}, for each $x\in X_Z$ there is some $w_x\in\RR^{d-2}$, $w_x\not =0$ such that the open ball of radius $|w_x|$ centred at $(t_x,x+w_x)$ is disjoint from $Z$. So we can apply Lemma~\ref{lemma_Baire} to find some $x\in X$ such that $x\not\in X_Z$. Then (using the final part of Lemma~\ref{lemma_largeinnerproducts}) we get that Lemma~\ref{lemma_pathsavoidingbadcase} applies for the path $\varphi\circ\gamma_x$ and hence $\varphi(\gamma_x(1))$ is reachable from $\varphi(\gamma_x(0))$ along a path which is $\epsilon$-close to $\varphi\circ\gamma_x$.
\end{proof}

For two points $x$ and $y$ in $\RR^{d-1}$, let  $\gamma_{x,y}$ denote the straight line segment from $x$ to $y$ (i.e., $\gamma_{x,y}(t)=(1-t)x+ty$ for $t\in [0,1]$). The following lemma is a more convenient version of Lemma~\ref{lemma_takingcharts}.

\begin{lemma}\label{lemma_chartsconvenient}
	Let $\varphi: \RR^{d-1}\to V$ be a smooth parametrization of some open set $V\subseteq \Sph^{d-1}$. Let $U_1, U_2$ be non-empty open subsets of $\RR^{d-1}$. Then there are some $x\in U_1, y\in U_2$ such that, for all $\epsilon>0$, $\varphi(y)$ is reachable from $\varphi(x)$ along a path which is $\epsilon$-close to $\varphi\circ \gamma_{x,y}$.
\end{lemma}

\begin{proof}
	We can take a bijective affine map $\psi: \RR^{d-1}\to\RR^{d-1}$ which maps $U_1$ to an open neighbourhood of $0$ and $U_2$ to an open neighbourhood of $(1,0,\dots,0)$. Then the statement follows easily from Lemma~\ref{lemma_takingcharts} applied to the parametrization $\varphi\circ \psi^{-1}$.
\end{proof}

We finish this subsection by giving the proofs of Lemmas~\ref{lemma_takingchartsnew} and \ref{lemma_Baire}.

\begin{proof}[Proof of Lemma~\ref{lemma_takingchartsnew}]
	Let $v=\varphi(t_x,x)$. By the definition of $Z$, we may find $u,w\in \RR^d$ such that $0<|w|<1$, $u,u+w\in I_v$, $\langle v,w\rangle=0$ and $\langle w,(\varphi\circ\gamma_x)'(t)\rangle =0$. By Lemma~\ref{lemma_spheresaroundtypeline}, the set $P=\{p\in \Sph^{d-1}: \langle p, v+w \rangle > 1 \}\cup \{p\in \Sph^{d-1}: \langle p, v-w \rangle > 1 \}$ has the property that for each $p\in P$, $I_p$ has non-empty interior. In particular, $\varphi(Z)$ is disjoint from $P$.
	
	By decreasing $|w|$ if necessary, we may assume that $P\subseteq V$. We want to show that for some $w'\in\RR^{d-2}$ ($w'\not =0$) the set $\varphi^{-1}(P)$ contains an open ball of radius $|w'|$ around $(t_x,x+w')$.
	
	Let $D$ be the derivative $D\varphi|_{(t_x,x)}$ of $\varphi$ at $(t_x,x)$, so $D$ is a bijective linear map $\RR^{d-1}\to\{v'\in\RR^d:\langle v,v'\rangle=0\}$. We can find an orthonormal basis $f_1,\dots,f_{d-2}$ of $\RR^{d-2}$ such that $f_1=(1,0,\dots,0)$, $\langle D(f_2), w\rangle>0$ and $\langle D(f_i), w\rangle=0$ for all $i\not =2$. Consider the ball of radius $\rho$ centred at $(t_x,x)+\rho f_2$. Any point of this open ball is of the form $q=(t_x,x)+\sum_{i=1}^{d-2}\lambda_if_i$ with $(\lambda_2-\rho)^2+\sum_{i\not =2}\lambda_i^2<\rho^2$. But we have
	\begin{align*}
	\varphi(q)&=v+\sum_{i=1}^{d-2}\lambda_i D(f_i)+O(\sum_{i=1}^{d-2}\lambda_i^2)
	\end{align*}
	and hence
	$$	\varphi(q)=v+\sum_{i=1}^{d-2}\lambda_i D(f_i)+O(2\rho \lambda_2).$$
	Using that $\langle v,D(f_i)\rangle=0$ for all $i$ and $\langle w,D(f_j)\rangle =0$ for all $j \not =2$, 
	
	\begin{align*}
	\langle v+w,\varphi(q)\rangle&=\langle v+w, v+\sum_{i=1}^{d-2}\lambda_i D(f_i)+O(2\rho \lambda_2)\rangle\\
	&=1+\lambda_2 \langle w,D(f_2)\rangle+O(2\rho\lambda_2).
	\end{align*}
	
	Since $\langle w, D(f_2)\rangle>0$, we get that there is some $\rho_0>0$ such that if $\rho\leq \rho_0$ then $\langle v+w,\varphi(q)\rangle>1$ (and hence $q\in\varphi^{-1}(P)$ and thus $q\not \in Z$) for all such points $q$. Since $f_1$ is orthogonal to $f_2$, we have $f_2=(0,y)$ for some $y\in \RR^{d-2}$, $|y| =1$. Then $w_x=\rho_0y$ satisfies the conditions.
\end{proof}

Before we formally prove Lemma~\ref{lemma_Baire}, let us give a sketch proof in the case when $d=3$, $X=(-1,1)$ and $w_x\in\RR$ is the same for all $x$: $w_x=r\in (0,1)$ for all $x\in X$. Assume that $X_Z=X$. Using that the circle of radius $r$ centred at $(x,t_x+r)$ does not contain $(y,t_y)$, it is easy to see that we must have $|t_y-t_x|\geq \Omega_r(\sqrt{y-x})$ whenever $0<y-x<r$ (see Figure~\ref{figure_Baire}). So if we take $N+1$ equally spaced points $x_0,\dots,x_N$ between $0$ and $r$ ($x_j=jr/N$), then $|t_{x_i}-t_{x_j}|\geq \Omega_r(1/\sqrt{N})$ for all $i,j$. It is easy to see that this gives a contradiction as $N\to\infty$. We will use the Baire category theorem to reduce the general case to a case similar enough to the one discussed above.
\begin{figure}[h!]
	\includegraphics[clip,trim=0cm 0cm 0cm 0cm, width=0.5\linewidth]{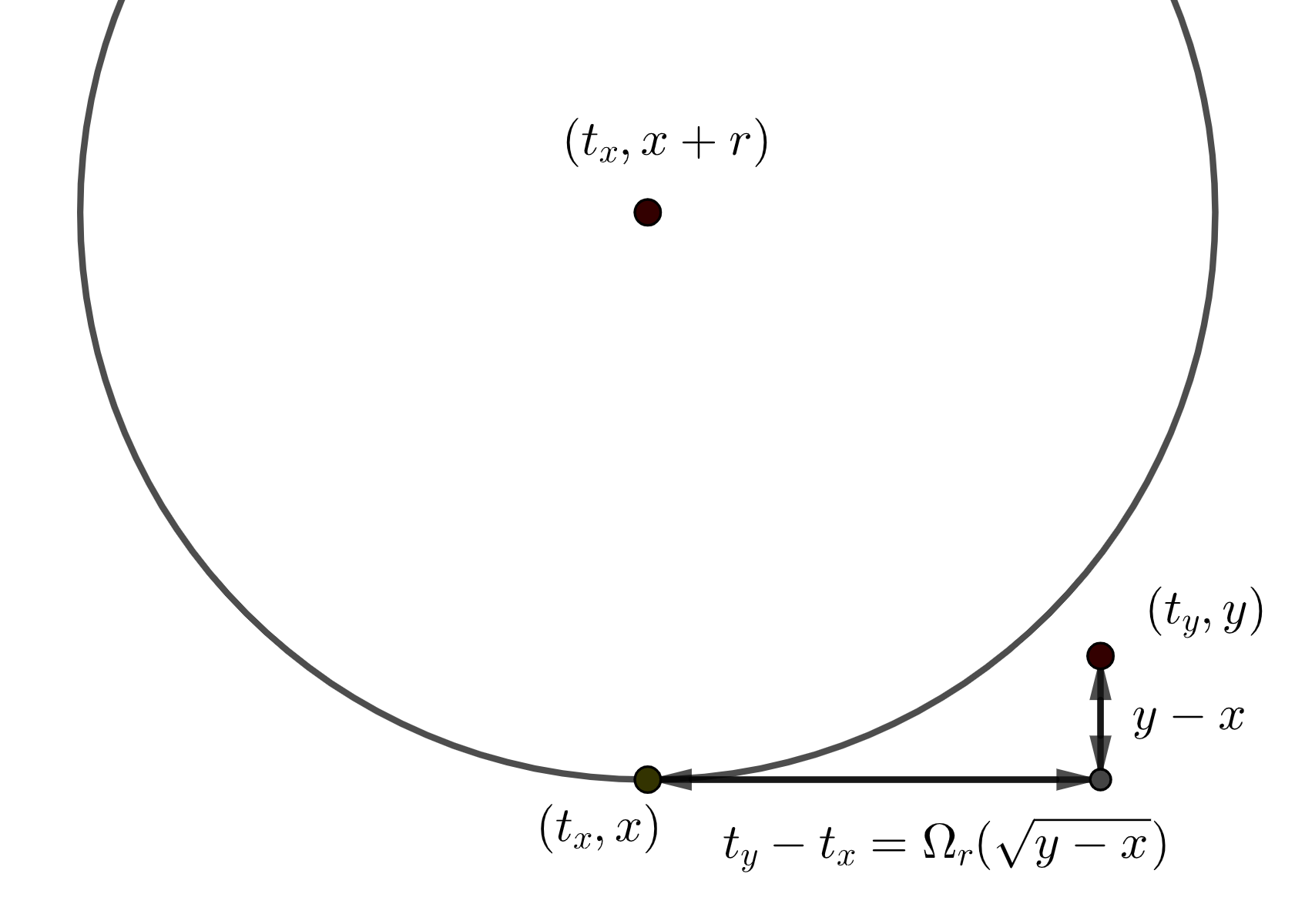}
	\centering
	\captionsetup{justification=centering}
	\caption{Since $(t_y,y)$ is not contained in the ball of radius $r$ centred at $(t_x,x+r)$, we have $|t_y-t_x|=\Omega_r(\sqrt{y-x})$.}
	\label{figure_Baire}
\end{figure}

\begin{proof}[Proof of Lemma~\ref{lemma_Baire}]
	For each positive integer $n$, let $X_Z^{n}=\{x\in X_Z: |w_x|\geq 1/n\}$. Clearly, $X_Z=\bigcup_n X_Z^n$. By the Baire category theorem, it is enough to show that each $X_Z^n$ is a finite union of nowhere dense sets. Assume, for contradiction, that $X_Z^n$ cannot be written as such a finite union. Let $\eta=1/4$, and for all $v\in \Sph^{d-3}$ let $U_v=\{u\in \Sph^{d-3}: \langle u,v\rangle>1-\eta\}$. Since $\Sph^{d-3}$ is compact, it is covered by finitely many such sets $U_v$. Write $Y_v=\{x\in X_Z^n: w_x/|w_x|\in U_v\}$. It follows that not every $Y_v$ is nowhere dense, i.e., there exist $v\in \Sph^{d-3}$, $y\in X$ and $\epsilon>0$ such that the closure of $Y_v$ contains all $x\in X$ with $|x-y|\leq \epsilon$. We may assume $\epsilon<1/n$. 
	Write $x(j)=y+\frac{j}{N}\epsilon v$ for $j=0,\dots, N$, where $N$ is some large positive integer (specified later). Note that $|x(j)-y|\leq \epsilon$ for all $j$, so there are some $y(j)\in Y_v$ such that $|x(j)-y(j)|<\eta/N^2$.

	\textbf{Claim.} If $0\leq i<j\leq N$ then $|t_{y(i)}-t_{y(j)}|=\Omega_{n,\epsilon}(1/N^{1/2})$.
	
	Note that if the claim holds, then $\max_i t_{y(i)}-\min_i t_{y(i)}=\Omega_{n,\epsilon}(N^{1/2})$. Then taking $N$ large enough gives a contradiction. So the lemma follows from the claim above.
	
	\textbf{Proof of Claim.} We will use that the open ball centred at $(t_{y(i)},y(i)+w_{y(i)})$ of radius $|w_{y(i)}|$ does not contain $(t_{y(j)},y(j))$. For simplicity, let us write $t_i$ for $t_{y(i)}$, $t_j$ for $t_{y(j)}$ and $w$ for $w_{y(i)}$. We may assume $|w|=1/n$. We have
	\begin{align*}
	|(t_i,y(i)+w)-(t_j,y(j))|^2&=|t_i-t_j|^2+|y(i)+w-y(j)|^2.
	\end{align*}
	But
	\begin{align*}
	|y(i)+w-y(j)|^2&=|w|^2+|y(i)-y(j)|^2-2\langle w, y(j)-y(i)\rangle\\
	&=|w|^2+|y(i)-y(j)|^2-2\langle w, x(j)-x(i)\rangle-2\langle w, y(j)-x(j)\rangle+2\langle w, y(i)-x(i)\rangle\\
	&\leq |w|^2+|y(i)-y(j)|^2-2\langle w, \frac{j-i}{N}\epsilon v\rangle +4\frac{\eta}{nN^2}\\
	&\leq |w|^2+ (|x(i)-x(j)|+2\eta/N^2)^2 -2\frac{j-i}{N}\epsilon(1-\eta)/n +4\frac{\eta}{nN^2}\\
	&=|w|^2+ \left(\frac{j-i}{N}\epsilon+2\eta/N^2\right)^2 -2\frac{j-i}{N}\epsilon(1-\eta)/n +4\frac{\eta}{nN^2}.
	\end{align*}
	But we know $|w|^2\leq |(t_i,y(i)+w)-(t_j,y(j))|^2$, thus
	\begin{align*}
	|t_i-t_j|^2&\geq 2\frac{j-i}{N}\epsilon(1-\eta)/n- \left(\frac{j-i}{N}\epsilon+2\eta/N^2\right)^2 -4\frac{\eta}{nN^2}\\
	&=\frac{2(j-i)\epsilon(1-\eta)}{n}\frac{1}{N}-\frac{(j-i)^2\epsilon^2}{N^2}-O_{n,\epsilon}(1/N^2)
	\end{align*}
	Using $\frac{(j-i)^2\epsilon^2}{N^2}\leq \frac{(j-i)\epsilon}{n}\frac{1}{N}$ (as $j-i\leq N$ and $\epsilon\leq 1/n$), we get
	\begin{align*}
	|t_i-t_j|^2&\geq \left(\frac{2(j-i)\epsilon(1-\eta)}{n}-\frac{(j-i)\epsilon}{n}\right)\frac{1}{N}-O_{n,\epsilon}(1/N^2).
	\end{align*}
	As we picked $\eta=1/4$, we get
	
	$$	|t_i-t_j|^2\geq \frac{(j-i)\epsilon}{2n}\frac{1}{N}(1-O_{n,\epsilon}(1/N)),$$
	and hence
	$$|t_i-t_j|\geq \left(\frac{\epsilon}{2n}\right)^{1/2}\frac{1}{N^{1/2}}(1-O_{n,\epsilon}(1/N)),$$
	proving the claim and hence the lemma.
\end{proof}

\subsection{Finishing the proof}

We now use our earlier lemmas (especially Lemma~\ref{lemma_chartsconvenient} and Lemma~\ref{lemma_spheresaroundtypeline}) to finish the proof of Theorem~\ref{theorem_reachable}.

\begin{lemma}\label{lemma_closedirectionseasycase}
	Let $\varphi: \RR^{d-1}\to V$ be a smooth parametrization of some open set $V\subseteq \Sph^{d-1}$, and let $\epsilon>0$. Assume that $v,v'\in V$ are such that $I_v, I_{v'}$ are non-empty. Then $v'$ is reachable from $v$ along a path which is $\epsilon$-close to $\varphi\circ\gamma_{\varphi^{-1}(v),\varphi^{-1}(v')}$.
\end{lemma}
\begin{proof}
	Write $u$ for $\varphi^{-1}(v)$ and $u'$ for $\varphi^{-1}(v')$. As $I_v, I_{v'}$ are non-empty, there is an open set containing $v$ and $v'$ such that whenever $p$ belongs to this set then $I_p$ has non-empty interior. By Lemma~\ref{lemma_nicetype}, there are open balls $U_1, U_2\subseteq \RR^{d-1}$ around $u$ and $u'$ (respectively) such that for any $x\in U_1$, $\varphi(x)$ is reachable from $v$ along a path which is $0$-close to $\varphi\circ\gamma_{u,x}$, and similarly for any $y\in U_2$, $\varphi(y)$ is reachable from $v'$ along a path which is $0$-close to $\varphi\circ\gamma_{u',y}$. Pick $\eta>0$ small (to be specified later). By Lemma~\ref{lemma_chartsconvenient}, we can find $x\in U_1$, $|x-u|<\eta$ and $y\in U_2$, $|y-u'|<\eta$ such that $\varphi(y)$ is reachable from $\varphi(x)$ along a path which is $(\epsilon/2)$-close to $\varphi\circ\gamma_{x,y}$. It follows that $u'$ is reachable from $u$ along a path which is $(\epsilon/2)$-close to the union of the images of $\varphi\circ\gamma_{u,x}$, $\varphi\circ\gamma_{x,y}$, $\varphi\circ\gamma_{y,u'}$. However, by taking $\eta$ small enough, we can guarantee that all points in these images are at most $\epsilon/2$ away from a point in the image of $\varphi\circ\gamma_{u,u'}$, proving the lemma.
\end{proof}

To extend Lemma~\ref{lemma_closedirectionseasycase} to all $v,v'$, including when $I_v$ or $I_{v'}$ is a single point, we will use the following lemma.

\begin{lemma}\label{lemma_typepoint}
	Let $\varphi: \RR^{d-1}\to V$ be a smooth parametrization of some open set $V\subseteq \Sph^{d-1}$. Assume that $v\in V$ and $I_v$ is a single point. Then one of the following statements hold.
	\begin{enumerate}
		\item For all $\eta>0$ there is some $p\in V$ such that $I_p$ has non-empty interior and $p$ is reachable from $v$ along a path which is $\eta$-close to $\{v\}$.
		\item There is some open neighbourhood $N$ of $v$ such that whenever $p\in N$ then $p$ is reachable from $v$ along $\varphi\circ\gamma_{\varphi^{-1}(v),\varphi^{-1}(p)}$.
	\end{enumerate}
\end{lemma}

\begin{proof}
	First, assume that there is a sequence of points $(p_n)$ in $V$ converging to $v$ such that for all $n$, $I_{p_n}$ is not a single point. We will show that the first conclusion holds. By Lemma~\ref{lemma_spheresaroundtypeline} (and the final part of Lemma~\ref{lemma_largeinnerproducts}), we may modify $p_n$ slightly so that $I_{p_n}$ has non-empty interior for all $n$. Let $\eta>0$ be given. By Lemma~\ref{lemma_closedirectionseasycase}, we can take $\gamma_n: [0,1]\to V$ and $\delta_n:[0,1]\to K$ continuous functions such that $\delta_n(t)\in I_{\gamma_n(t)}$ for all $(n,t)$, $\gamma_n(0)=p_n$ for all $n$, $\gamma_n(1)=p_{n+1}$ for all $n$, and $\gamma_n(t)$ is at most $\eta/2^n$ away from some point on the image of $\varphi\circ\gamma_{\varphi^{-1}(p_n),\varphi^{-1}(p_{n+1})}$ for all $(n,t)$. By taking a subsequence of the form $(p_n)_{n>N_0}$, we may assume that for all $n$, all 
	points on the image of $\varphi\circ\gamma_{\varphi^{-1}(p_n),\varphi^{-1}(p_{n+1})}$ are at most $\eta/2$ away from $v$. So $|\gamma_n(t)-v|\leq \eta$ for all $(n,t)$. Using Lemma~\ref{lemma_canchooseendpoints}, we may also assume that $\delta_n(1)=\delta_{n+1}(0)$ for all $n$.
	
	Now define $\gamma: [0,1]\to V$ and $\delta: [0,1]\to K$ as follows. Let $\gamma(0)=v$ and let $\delta(0)$ be the unique point in $I_v$. For $t\in (0,1]$, let $n$ be such that $\frac{1}{n+1}\leq t\leq \frac{1}{n}$, and set $\gamma(t)=\gamma_n(n(n+1)(\frac{1}{n}-t))$ and $\delta(t)=\delta_n(n(n+1)(\frac{1}{n}-t))$. It is easy to check that $\gamma, \delta$ are well-defined and continuous on $(0,1]$, and $|\gamma(t)-v|\leq \eta$ for all $t$. Moreover, using that $(p_n)\to v$ and $\gamma_n(t)$ is at most $\eta/2^n$ away from some point on the image of $\varphi\circ\gamma_{\varphi^{-1}(p_n),\varphi^{-1}(p_{n+1})}$ for all $(n,t)$, we also get that $\gamma$ is continuous at $0$. To show continuity of $\delta$ at $0$, assume that $(t_n)\to 0$ and $(\delta(t_n))\to z$, we prove $z=\delta(0)$. We know $\delta(t_n),\delta(t_n)+\gamma(t_n)\in K$. Using that $K$ is closed and $\gamma$ is continuous, taking limits gives $z, z+\gamma(0)\in K$, i.e., $z, z+v\in K$, i.e., $z\in I_v$. Hence $z=\delta(0)$, as claimed. This proves the claim in the first case.
	
	Now assume that such a sequence $(p_n)$ does not exist. This means that there is an open neighbourhood of $v$ consisting only of points $p$ such that $I_p$ is a single point. It follows that there is an open ball $B$ around $u=\varphi^{-1}(v)$ such that whenever $x\in B$ then $I_{\varphi(x)}$ is a single point. Let $N=\varphi(B)$, so $N$ is an open neighbourhood of $v$. Given $p\in N$, let $\varphi^{-1}(p)=q$. We show $p$ is reachable from $v$ along $\varphi\circ\gamma_{u,q}$. Indeed, let $\gamma(t)=\varphi((1-t)u+tq)$ and let $\delta(t)$ be the unique point in $I_{\gamma(t)}$. Then $\delta$ is continuous by an argument almost identical to the one above. Indeed, if $(t_n)\to t$ and $(\delta(t_n))\to z$, then $\delta(t_n), \delta(t_n)+\gamma(t_n)\in K$. Taking limits gives $z,z+\gamma(t)\in K$, i.e., $z\in I_{\gamma(t)}$, i.e., $z=\delta(t)$, as required. This finishes the proof of the lemma.
\end{proof}

\begin{lemma}\label{lemma_closedirectionswork}
	Let $\varphi: \RR^{d-1}\to V$ be a smooth parametrization of some open set $V\subseteq \Sph^{d-1}$, and let $\epsilon>0$. Then for any $v,v'\in V$, $v'$ is reachable from $v$ along a path which is $\epsilon$-close to $\varphi\circ\gamma_{\varphi^{-1}(v),\varphi^{-1}(v')}$.
\end{lemma}
\begin{proof}
	Write $u$ for $\varphi^{-1}(v)$ and $u'$ for $\varphi^{-1}(v')$.
	Let $\eta>0$ be small (specified later). There is some open set $V_1\subseteq V$ (not necessarily containing $v$) such that any $p\in V_1$ is reachable from $v$ along a path which is $\eta$-close to $\{v\}$. Indeed, this follows from Lemma~\ref{lemma_spheresaroundtypeline} if $I_v$ is not a single point, and from Lemma~\ref{lemma_typepoint} (together with Lemma~\ref{lemma_nicetype}) when $I_v$ is a single point. Similarly, there is some $V_2\subseteq V$ such that any $q\in V_2$ is reachable from $v'$ along a path which is $\eta$-close to $\{v'\}$. In particular, $|v-p|\leq \eta$ and $|v'-q|\leq \eta$ for any such $p,q$.
	
	But, by Lemma~\ref{lemma_chartsconvenient}, there are some $p\in V_1, q\in V_2$ such that $q$ is reachable from $p$ along a path which is $\eta$-close to $\varphi\circ\gamma_{\varphi^{-1}(p),\varphi^{-1}(q)}$. Hence $v'$ is reachable from $v$ along a path which is $\eta$-close to $\{v,v'\}\cup\operatorname{Im}(\varphi\circ\gamma_{\varphi^{-1}(p),\varphi^{-1}(q)})$. By taking $\eta$ small enough, we can guarantee that any point in $\operatorname{Im}(\varphi\circ\gamma_{\varphi^{-1}(p),\varphi^{-1}(q)})$ is at most $\epsilon/2$ away from some point in $\operatorname{Im}(\varphi\circ\gamma_{\varphi^{-1}(v),\varphi^{-1}(v')})$. The result follows.
\end{proof}

\begin{proof}[Proof of Theorem~\ref{theorem_reachable}]
	Using Lemma~\ref{lemma_closedirectionswork}, it is easy to see that for any $t\in [0,1]$ there is some $\delta_t>0$ such that whenever $t'\in [0,1]$ and $|t-t'|<\delta_t$, then $\gamma(t')$ is reachable from $\gamma(t)$ along a path which is $\epsilon$-close to $\{\gamma(t)\}$. The result follows easily (using the compactness of $[0,1]$ and Lemma~\ref{lemma_canchooseendpoints}).
\end{proof}
\begin{proof}[Proof of Theorem~\ref{theorem_Kakeyasegments}]
	The result follows immediately from Theorem~\ref{theorem_reachable} when $d\geq 3$, and from Theorem~\ref{theorem_continuousin2d} when $d=2$ (using Lemma~\ref{lemma_canchooseendpoints}, which also holds for $d=2$).
\end{proof}

\section{Counterexample for general bodies}\label{section_reachablecounterexample}

In this section, our goal is to prove Theorem~\ref{theorem_counterexamplereachable_introduction}, restated below for convenience.
\fourdimcounterexample*
	We will use similar ideas as for Theorem~\ref{theorem_continuouscounterexample} (but this proof will be significantly more complicated). Note that it is sufficient to find a construction where $S$ is compact but not necessarily convex, as the same set $K$ will still provide a counterexample when $S$ is replaced by its convex hull. The set $S$ in our construction will be given by $$S=\{(x,y,z,w)\in\RR^4: x^2+y^2+z^2+w^2=1, x=\pm 1/2\},$$
	see Figure~\ref{figure_counterexample}. 
	
	\begin{figure}[h!]
		\includegraphics[clip,trim=0cm 0cm 0cm 0cm, width=0.3\linewidth]{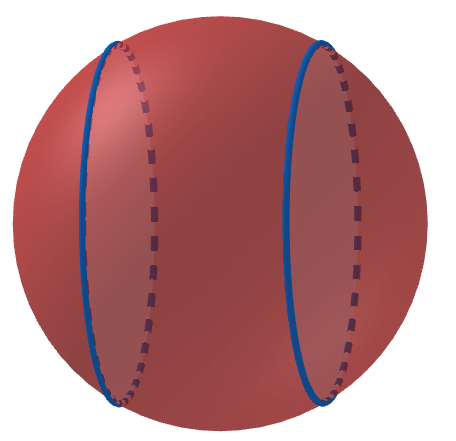}
		\centering
		\captionsetup{justification=centering}
		\caption{The set $S$ in our construction is a $4$-dimensional analogue of the blue set (or the convex hull of the blue set), which is a subset of the (red) unit sphere}\label{figure_counterexample}
	\end{figure}
In the proof of Theorem~\ref{theorem_continuouscounterexample} we made sure that our set $K$ lied inside the cylinder $\{(x,y,z):y^2+z^2\leq 1/4\}$, and we controlled the intersection with the boundary of the cylinder. This control enabled us to prove discontinuity by observing that any segment in a direction of the $yz$ plane had to intersect the boundary of the cylinder in a pair of points $(x,y,z), (x,-y,-z)$.

We will attempt to do something similar here. Our construction will be contained inside the set $\{(x,y,z,w):x^2+y^2\leq 1\}$, and we will control the intersection with the boundary $C=\{(x,y,z,w)\in\RR^4:x^2+y^2=1\}$ of that set. Observe that any rotated copy of $S$ is of the form $$S_v=\{v'\in\RR^4: |v'|=1, \langle v,v'\rangle=\pm 1/2\}$$ for some $v\in \RR^4$ with $|v|=1$. It is not difficult to deduce that if we only rotate $S$ slightly, then the rotated copy intersects $C$ in two pairs of antipodal points. (See Figure~\ref{figure_counterexample}: great circles close to the one given by $x^2+y^2=1, z=0$ intersect the blue set in two pairs of antipodal points). We will have to make sure that $K$ contains translated copies of any two such pairs of antipodal points (so that a translate of $\rho(S)$ is contained in $K$ for all $\rho$), so for all $(x_1,y_1),(x_2,y_2)\in \Sph^1$ we will have some $(z,w)$ such that $(\pm(x_1,y_1),z,w),(\pm(x_2,y_2),z,w)\in C\cap K$. Meanwhile, we will have restrictions on $C\cap K$ in such a way that we guarantee discontinuity.\medskip

Let us now turn to the formal proof of Theorem~\ref{theorem_counterexamplereachable_introduction}. As mentioned above, we will control the intersection of $K$ with $C$, i.e., for all $(x,y)\in \Sph^1$ we will control the set $A_{x,y}=\{(z,w):(x,y,z,w)\in K\}$. The following lemma lists all the properties that we will need -- for now, we only show that such sets $A_{x,y}$ exist in $\RR^2$, at this point they do not necessarily come from a body $K$ in $\RR^4$.

\begin{lemma}\label{lemma_cylinderproperties}
	There exist compact convex sets $(A_p)_{p\in \Sph^1}$ in $\RR^2$ such that the following properties hold.
	\begin{enumerate}
		\item For all $p,q\in \Sph^1$, $A_p\cap A_q\not =\emptyset$.
		\item For all $p\in \Sph^1$, $A_p=A_{-p}$.
		\item For all $p\in \Sph^1$ and all $t\in A_p$ we have $|t|\leq 1$.
		\item The set $\{(p,t):p\in \Sph^1, t\in A_p\}$ is closed, i.e, whenever $(p_n)\to p$ in $\Sph^1$ and $(t_n)\to t$ in $\RR^2$ with $t_n\in A_{p_n}$ for all $n$, then $t\in A_p$.
		\item For all $\epsilon>0$ and $(z,w)\in \RR^2$ there is some $r\in (0,\epsilon)$ such that whenever $p=(x,y)\in \Sph^1$ with $|x-1/2|=r$ then all points of $A_p$ are at least distance $1/100$ away from $(z,w)$.
	\end{enumerate}
\end{lemma}

Note that such sets $A_p$ cannot exist in $\RR$ instead of $\RR^2$: each $A_p$ would have to be a non-empty closed bounded interval, and then $\bigcap_{p\in \Sph^1}A_p$ would be non-empty by the first condition, so the last property could not be satisfied. This is the reason we need $\RR^4$ for our construction instead of $\RR^3$.

Before we prove Lemma~\ref{lemma_cylinderproperties}, we state two lemmas which show why it is useful: Theorem~\ref{theorem_counterexamplereachable_introduction} will follow immediately from Lemma~\ref{lemma_cylinderproperties} and these lemmas. Recall that $C=\Sph^1\times \RR^2$ and $S_v=\{v'\in\RR^4: |v'|=1, \langle v,v'\rangle=\pm 1/2\}$. 

\begin{lemma}\label{lemma_constructionexists}
	Assume that we have compact convex sets $(A_p)_{p\in \Sph^1}$ in $\RR^2$ such that the following properties hold.
	\begin{enumerate}
		\item For all $p,q\in \Sph^1$, $A_p\cap A_q\not =\emptyset$.
		\item For all $p\in \Sph^1$, $A_p=A_{-p}$.
		\item For all $p\in \Sph^1$ and all $t\in A_p$ we have $|t|\leq 1$.
		\item The set $\{(p,t):p\in \Sph^1, t\in A_p\}$ is closed, i.e, whenever $(p_n)\to p$ in $\Sph^1$ and $(t_n)\to t$ in $\RR^2$ with $t_n\in A_{p_n}$ for all $n$, then $t\in A_p$.
	\end{enumerate}
	Then there exists a compact convex $S$-Kakeya set $K\subseteq \RR^4$ such that $K\subseteq \{(x,y,z,w)\in \RR^4:x^2+y^2\leq 1\}$ and $K\cap C\subseteq \{(p,t):p\in \Sph^1, t\in A_p\}$.
\end{lemma}

\begin{lemma}\label{lemma_constructionworks}
	Assume that $(A_p)_{p\in \Sph^1}$ in $\RR^2$ are compact convex sets such that the following property holds: for all $\epsilon>0$ and $(z,w)\in \RR^2$ there is some $r\in (0,\epsilon)$ such that whenever $p=(x,y)\in \Sph^1$ with $|x-1/2|=r$ then all points of $A_p$ are at least distance $1/100$ away from $(z,w)$. Assume furthermore that $K$ is a compact convex set such that $K\subseteq \{(x,y,z,w)\in \RR^4:x^2+y^2\leq 1\}$ and $K\cap C\subseteq \{(p,t):p\in \Sph^1, t\in A_p\}$. Then whenever $\gamma: [0,1]\to \Sph^3$ and $\delta: [0,1]\to \RR^4$ are continuous such that $\gamma(0)=(1,0,0,0)$ and $S_{\gamma(t)}+\delta(t)\subseteq K$ for all $t$, then $\gamma(t)=(1,0,0,0)$ for all $t$.
\end{lemma}

We now prove Lemmas~\ref{lemma_cylinderproperties}, \ref{lemma_constructionexists} and \ref{lemma_constructionworks}.
\begin{proof}[Proof of Lemma~\ref{lemma_cylinderproperties}]
		Consider the following four sets in $\RR^2$:
		\begin{align*}
		T_1&=\{0\}\times [0,1],\\
		T_2&=[0,1]\times\{0\},\\
		T_3&=\{(z,w)\in\RR^2: z+w=1, 0\leq z,w\leq 1\},\\
		T&=\{(z,w)\in\RR^2: 0\leq z,w\leq 1, 0\leq z+w\leq 1\}.
		\end{align*}
		Given $(x,y)$ with $x^2+y^2=1$, we define $A_{x,y}$ as follows. Let $\min(|x-1/2|,|x+1/2|)=s$. If $s=0$, then $A_{x,y}=T$. Otherwise, let $k$ be the positive integer such that $1/2^k\geq s>1/2^{k+1}$. If $s=1/2^k$, then let $A_{x,y}=T$. Otherwise let $A_{x,y}=T_{k\textnormal{ mod }3}$.
		
		It is straightforward to check that each $A_p$ is convex and compact, and that properties 1, 2 and 3 are satisfied. To see that property 4 holds, observe that if $(p_n)\to p$ and $(t_n)\to t$ as above, then either $A_p=T$, or $A_{p_n}$ is eventually constant and equal to $A_p$. In either case, it is easy to deduce that $t\in A_p$.
		
		Finally, we show that property 5 holds. Given such $(z,w)$, we can find some $i\in\{1,2,3\}$ such that any point in $T_i$ has distance at least $1/100$ from $(z,w)$. Then we can find some $r\in (0,\epsilon)$ such that $1/2^k> r>1/2^{k+1}$ for some positive integer $k$ with $k\equiv i\textnormal{ mod $3$}$. It is easy to see that this $r$ satisfies the conditions.
\end{proof}

\begin{proof}[Proof of Lemma~\ref{lemma_constructionexists}]
	Observe that if $v\in \Sph^3$, the set $S_v$ intersects $C$ in $0$, $2$ or $4$ points:
	\begin{itemize}
		\item $S_v$ intersects $C$ in $0$ points if and only if $v_1^2+v_2^2<1/4$;
		\item $S_v$ intersects $C$ in a pair of points $v', -v'$ if and only if $v_1^2+v_2^2=1/4$;
		\item $S_v$ intersects $C$ in two pairs of (distinct) points $v', -v', v'', -v''$ if and only if $v_1^2+v_2^2>1/4$.
	\end{itemize}
	
	Let 
	\begin{align*}
	V_1&=\{v\in \Sph^3: v_1^2+v_2^2\geq 1/4\},\\
	V_2&=\{v\in \Sph^3: 1/100\leq v_1^2+v_2^2\leq 1/4\},\\
	V_3&=\{v\in \Sph^3: v_1^2+v_2^2\leq 1/100\}.
	\end{align*}
	
	For all $v\in V_1$, let $T_v=\bigcap_{w\in C\cap S_v}A_{w_1,w_2}=\bigcap_{p\in \Sph^1: (p,0,0)\in S_v}A_{p}$. Note that $C\cap S_v=\{v',-v',v'',-v''\}$ for some $v',v''\in \Sph^3$ (not necessarily distinct), so (using $A_{-p}=A_{p}$) we have $T_v=A_{v'_1,v'_2}\cap A_{v''_1,v''_2}$. In particular, $T_v\not =\emptyset$. Let $K_{1,v}=S_v+\{(0,0,t):t\in T_{v}\}$ and $$K_1=\bigcup_{v\in V_1}K_{1,v}.$$

	For all $v\in V_2$, let $p_v=\left(\frac{1/2}{\sqrt{v_1^2+v_2^2}}v_1,\frac{1/2}{\sqrt{v_1^2+v_2^2}}v_2, \frac{\sqrt{3}/2}{\sqrt{v_3^2+v_4^2}}v_3,\frac{\sqrt{3}/2}{\sqrt{v_3^2+v_4^2}}v_4\right)$. So $(p_v)_1^2+(p_v)_2^2=1/4$, $|p_v|=1$, and we have $p_v=v$ if $v_1^2+v_2^2=1/4$. Let $K_{2,v}=S_v+\{(0,0,t):t\in T_{p_v}\}
	$. (Note that $C\cap S_{p_v}=\{v',-v'\}$, where $v'=2((p_v)_1,(p_v)_2,0,0)$ and hence $T_{p_v}= A_{2(p_v)_1,2(p_v)_2}$.) Let
	$$K_2=\bigcup_{v\in V_2}K_{2,v}.$$
	
	For all $v\in V_3$, let $K_{3,v}=S_v$, and let 
	$$K_3=\bigcup_{v\in V_3}K_{3,v}.$$
	
	Finally, let $K_0=K_1\cup K_2\cup K_3$, and let $K$ be the convex hull of $K_0$.\medskip
	
	\textbf{Claim.} The set $K_0$ has the following properties.
	\begin{enumerate}
		\item For each $v\in \Sph^3$ there is some $w\in \RR^4$ such that $S_v+w\subseteq K_0$.
		\item We have $K_0\cap C\subseteq \{(p,t):p\in \Sph^1, t\in A_p\}$, and $K_0$ has no point $(x,y,z,w)$ with $x^2+y^2>1$.
		\item The set $K_0$ is compact.
	\end{enumerate}
Note that these properties are preserved when taking convex hull. So the claim above implies the statement of the lemma.\medskip
	
	\textbf{Proof of Claim.} The first property holds because $K_{i,v}$ contains a translate of $S_v$ if $v\in V_i$. To see that the second property holds, observe that $K_0$ is a union of sets of the form $S_v+(0,0,t)$ for some $t\in \RR^2$. It follows that $K_0$ has no point $(x,y,z,w)$ with $x^2+y^2>1$. Also, if $(p,t)\in K_0\cap C$ ($p\in \Sph^1, t\in \RR^2)$, then $(p,t)\in S_v+(0,0,t)$ for some $v\in \Sph^3$ having $v_1^2+v_2^2\geq 1/4$, and $t\in T_v$ and $(p,0,0)\in C\cap S_v$. But $(p,0,0)\in C\cap S_v$ implies $T_v\subseteq A_p$, so $t\in A_p$, as claimed. It is easy to see that $K_0$ is bounded, so the only property left to check is that $K_0$ is closed. It is enough to show that $K_1, K_2, K_3$ are all closed.
	
	We first show that $K_3$ is closed. Assume that $(q_n)$ is a sequence of points in $K_3$ with $(q_n)\to q$, we show that $q\in K_3$. We know $q_n\in S_{v(n)}$ for some $v(n)\in V_3$. By taking an appropriate subsequence, we may assume that $v(n)$ converges to some $v\in V_3$. It is easy to see that $q\in S_v$ must hold, so then $q\in K_3$.
	
	Next, we show that $K_2$ is closed. As before, assume that $(q_n)$ is a sequence of points in $K_2$ with $(q_n)\to q$. We have $q_n\in S_{v(n)}+(0,0,t_n)$ for some $v(n)\in V_2$ and $t_n\in T_{p_{v(n)}}=A_{2(p_{v(n)})_1,2(p_{v(n)})_2}$. By taking a subsequence, we may assume that $v(n)$ converges to some $v\in V_2$, and $(t_n)$ converges to some $t\in \RR^2$. Observe that $(p_{v(n)})\to p_v$. But then $t\in A_{2(p_v)_1,2(p_v)_2}=T_{p_v}$ and hence $q\in S_v+(0,0,t)\subseteq S_v+\{(0,0,t'):t'\in T_{p_v}\}$, so $q\in K_{2}$, as required.
	
	Finally, we show that $K_1$ is also closed. Again, assume that $(q_n)$ is a sequence of points in $K_1$ with $(q_n)\to q$. We have $q_n\in S_{v(n)}+(0,0,t_n)$ for some $v(n)\in V_1$ and $t_n\in T_{v(n)}$. As before, by taking a subsequence we may assume that $v(n)$ converges to some $v\in V_1$ and $t_n$ converges to some $t\in \RR^2$. We claim that this implies $t\in T_v$. Observe that $C\cap S_{v(n)}$ is of the form $\{v'(n),-v'(n),v''(n),-v''(n)\}$, where $v'(n)=\pm v''(n)$ if and only if $v(n)_1^2+v(n)_2^2=1/4$. So we have
	$$T_{v(n)}=A_{v'(n)_1,v'(n)_2}\cap A_{v''(n)_1,v''(n)_2}.$$
	
	By taking an appropriate subsequence, we may assume that $v'(n)$ converges to $v'$ and $v''(n)$ converges to $v''$, where $C\cap S_v=\{v',-v',v'',-v''\}$. But we have $t_n\in A_{v'(n)_1,v'(n)_2}$ for all $n$, and hence $t\in A_{v'_1,v'_2}$. Similarly, $t\in A_{v''_1,v''_2}$. Hence $t\in T_v$, as claimed. But then $$q\in S_v+(0,0,t)\subseteq S_v+\{(0,0,t'):t'\in T_v\}=K_{1,v}\subseteq K_1,$$ as claimed. This finishes the proof of the claim and hence the lemma.
\end{proof}

\begin{proof}[Proof of Lemma~\ref{lemma_constructionworks}]
		Assume, for contradiction, that $\gamma(t)\not=(1,0,0,0)$ for some $t$. We may assume that $\gamma(t)_1>9/10$ for all $t$, and that for all $t>0$ we have $\gamma(t)\not =(1,0,0,0)$. There are some continuous functions $v', v'': [0,1]\to C$ such that $S_{\gamma(t)}\cap C=\{v'(t),-v'(t),v''(t),-v''(t)\}$, $\langle \gamma(t), v'(t)\rangle=\langle \gamma(t),v''(t)\rangle=1/2$ and $v'(0),v''(0)=(1/2,\pm\sqrt{3}/2,0,0)$.
	
	Observe that if $\gamma(t)\not =(1,0,0,0)$ then $v'(t)_1\not=1/2$ or $v''(t)_1\not =1/2$. Indeed, we would have $v'(t),v''(t)=(1/2,\pm\sqrt{3}/2,0,0)$ and $1=\langle \gamma(t),v'(t)+v''(t)\rangle=\langle \gamma(t),(1,0,0,0)\rangle$, giving $\gamma(t)=(1,0,0,0)$. It follows that for all $t>0$, either $v'(t)_1\not=1/2$ or $v''(t)_1\not =1/2$.
	
	By continuity, there is some $\epsilon>0$ such that for all $t\leq\epsilon$ we have $|\delta(t)-\delta(0)|<1/100$. We know $v'(\epsilon)_1\not=1/2$ or $v''(\epsilon)_1\not =1/2$, we may assume by symmetry that $v'(\epsilon)_1\not=1/2$. By assumption, there is an $x_0$ lying between $1/2$ and $v'(\epsilon)_1$ such that whenever $p\in \Sph^1$ is of the form $p=(x_0,y_0)$ (for some $y_0$) then any point of $A_p$ is at least distance $1/100$ away from $(\delta(0)_3,\delta(0)_4)$. But, by continuity of $v'$, there is some $t_0\in [0,\epsilon]$ such that $v'(t_0)_1=x_0$. Observe that
	$$K\supseteq S_{\gamma(t_0)}+\delta(t_0)\supseteq \{v'(t_0),-v'(t_0)\}+\delta(t_0).$$
	But if $u,u'\in K$ with $u-u'=2(x,y,0,0)$ for some $x,y$ with $x^2+y^2=1$, then we must have $u,u'\in K\cap C$ and $u=(x,y,z,w)$, $u'=(-x,-y,z,w)$ for some $(z,w)\in A_{x,y}$. Hence $\delta(t_0)=(0,0,z,w)$ for some $(z,w)\in A_{v'(t_0)_1,v'(t_0)_2}$. But then $|\delta(t_0)-\delta(0)|>1/100$, giving a contradiction.
\end{proof}

\begin{proof}[Proof of Theorem~\ref{theorem_counterexamplereachable_introduction}]
	The result follows easily from Lemmas~\ref{lemma_cylinderproperties}, \ref{lemma_constructionexists} and \ref{lemma_constructionworks}.
\end{proof}

\section{Concluding remarks}\label{section_Kakeyaconcluding}

In this paper we answered Question~\ref{question_kakeya} and some related problems. However, there are still some open questions in this topic. For example, our counterexample in Theorem~\ref{theorem_counterexamplereachable_introduction} requires $d\geq 4$, whereas we know that there can be no $2$-dimensional counterexample (by Theorem~\ref{theorem_continuousin2d}). It would be interesting to see a counterexample in $3$ dimensions (we believe that such a construction should exist).

\begin{question}
Can we find convex bodies $S$ and $K$ in $\RR^3$ such that $S$ is $K$-Kakeya, but there are two $S$-copies in $K$ which cannot be rotated into each other within $K$?
\end{question}

Furthermore, we showed that if $S$ is a unit segment, then any two $S$ copies can be rotated into each other within a compact convex ($S$-)Kakeya set, but this fails for general bodies $S$. It would be interesting to determine if there are other sets $S$ (or families of such) for which this property holds. (A trivial example is given by closed balls.)

\begin{question}
	Can we find (compact, convex) sets $S$ in $\RR^d$ with $d\geq 3$ such that $S$ is not a segment or a ball, and whenever some convex body $K$ is $S$-Kakeya then any two $S$ copies can be rotated into each other within $K$?
\end{question}

\bibliography{Bibliography}
\bibliographystyle{abbrv}	

\end{document}